\documentclass[reqno]{amsart}

\usepackage{amsfonts}
\usepackage{amssymb}
\usepackage{amsmath}
\usepackage{enumitem}
\usepackage{xcolor}

\setenumerate{label={\rm (\alph{*})}}
\usepackage{bbm,euscript,mathrsfs}
\usepackage{paper_diening}
\usepackage{graphicx}
\usepackage{hyperref}
\usepackage[top=1in, bottom=1.25in, left=1.10in, right=1.10in]{geometry}

\allowdisplaybreaks

\usepackage{tikz-cd}
\usetikzlibrary{lindenmayersystems}
\usetikzlibrary{decorations.pathreplacing}

\DeclareMathOperator{\Bog}{Bog}

\DeclareMathOperator{\Div}{div}

\newcommand{\R}{\mathbb R}
\newcommand{\N}{\mathbb N}

\newcommand{\dt}{\,\mathrm{d} t}

\newcommand{\dx}{\,\mathrm{d}x}
\newcommand{\dz}{\,\mathrm{d}z}

\newcommand{\dy}{\,\mathrm{d}y}

\newcommand{\dxt}{\,\mathrm{d}x\,\mathrm{d}t}
\newcommand{\ds}{\,\mathrm{d}\sigma}
\newcommand{\dxs}{\,\mathrm{d}x\,\mathrm{d}\sigma}

\newcommand{\mt}{\Omega}

\newtheorem{theorem}{Theorem}[section]
\newtheorem{lemma}[theorem]{Lemma}

\newtheorem{corollary}[theorem]{Corollary}
\newtheorem{remark}[theorem]{Remark}

\theoremstyle{definition}
\newtheorem{definition}[theorem]{Definition}

\begin{document}

\title[Boundary regularity Navier--Stokes]
{Partial boundary regularity for the Navier--Stokes equations in time-dependent domains}

\author{Dominic Breit}
\address{Institute of Mathematics, TU Clausthal, Erzstra\ss e 1, 38678 Clausthal-Zellerfeld, Germany}
\email{dominic.breit@tu-clausthal.de}




\date{\today}



\begin{abstract}
We consider the incompressible Navier--Stokes equations in a moving domain whose boundary is prescribed by a function $\eta=\eta(t,y)$ (with $y\in\R^2$) of low regularity. This is motivated by problems from fluid-structure interaction.
We prove partial boundary regularity for boundary suitable weak solutions assuming that $\eta$ is continuous in time with values in the fractional Sobolev space $W^{2-1/p,p}_y$ for some $p>15/4$ and we have $\partial_t\eta\in L_t^{3}(W^{1,q_0}_y)$ for some $q_0>2$.

The existence of boundary suitable weak solutions is a consequence of a new maximal regularity result for the Stokes equations in moving domains which is of independent interest.
\end{abstract}

\maketitle

\section{Introduction}
We consider the incompressible Navier--Stokes equations
 \begin{align} \partial_t \bfu+(\bfu\cdot\nabla)\bfu&=\Delta\bfu-\nabla \pi+\bff,\quad \Div\bfu=0,&\label{1}
 \end{align}
in a moving domain $\Omega_{\eta(t)}$ whose boundary is prescribed by a function $\eta$ which deforms the boundary of the reference domain $\Omega\subset\R^3$ in normal direction. 
For technical simplification we suppose that $\eta$ is defined on $\mathcal I \times \omega$ where $\mathcal I=(0,T)$ for some given $T>0$ and we identify $\partial\Omega$ with the two-dimensional torus denoted by $\omega$.
Velocity field $\bfu$ and pressure $\pi$ in \eqref{1} are defined in the deformed space-time cylinder
$$\mathcal I\times\Omega_\eta:=\bigcup_{t\in \mathcal I}\set{t}\times\Omega_{\eta(t)}\subset\R^{1+3}.$$

 This is motivated by applications in fluid-structure interaction on which we comment more below.
The quantity $\bff:\mathcal I\times\Omega_\eta\rightarrow \R^3$ in \eqref{1} is an external forcing.
 
\textbf{Interior regularity}. The question about the interior regularity of solutions to \eqref{1} (for which it does not matter if the boundary is moving) is a big open problem and the best known result is partial regularity, i.e., that the velocity field is locally bounded/H\"older continuous outside a negligible set of the space-time cylinder (further regularity properties inside this set can be deduced) with measure zero. The first result in this direction was proved in \cite{Sh1}. Eventually, in \cite{CKN} the understanding was deepened by introducing so-called suitable weak solutions, which satisfy 
the local energy inequality
 \begin{equation}\label{energylocal0a}
\begin{split}
\int\frac{1}{2}&\zeta\big| \bfu(t)\big|^2\dx+\int_0^t\int\zeta|\nabla \bfu |^2\dx\ds\\& \leq\int_0^t\int\frac{1}{2}\Big(| \bfu|^2(\partial_t\zeta+\Delta\zeta)+\big(|\bfu|^2+2\pi\big)\bfu\cdot\nabla\zeta\Big)\dx\ds+\int_0^t\int\zeta\bff\cdot\bfu\dxs
\end{split}
\end{equation}
for any function $\zeta$ (which is smooth, compactly supported and non-negative). This enabled a local regularity theory, see also \cite{CL,LaSe,L}.

\textbf{Fixed domains}.
The results mentioned so far concern the interior regularity of solutions. In the case of flat boundaries the regularity analysis at boundary points has ben started in \cite{Sh4} and continued further, e.g. in \cite{Se1} and \cite{Se3}. First results for smooth curved boundaries are shown in \cite{SeShSo}.
The question about partial regularity in the case of irregular domains has been addressed only very recently in  \cite{Br2}, where the author showed a corresponding result assuming that the local charts in the parametrisation of the boundary belong to the fractional Sobolev space
\begin{align}\label{ass:main}  
W^{2-\frac{1}{p},p}(\R^2)\quad\text{for some}\quad p>\tfrac{15}{4}.
\end{align}
The analysis in \cite{SeShSo} and \cite{Br2} is based on boundary suitable weak solutions satisfying a version of \eqref{energylocal0a} for boundary points and the local transformation of the momentum equation to a flat geometry. By employing Sobolev multipliers, non-smooth domains (leading to non-smooth coefficients in the transformed equation) can be included in the theory in \cite{Br2}.

 \textbf{Moving domains}. The purpose of this paper is to develop a partial boundary regularity theory for solutions to \eqref{1} in moving domains under low regularity assumptions on the function $\eta$ describing the moving boundary.
In a first step we prove the existence of boundary suitable weak solutions, where the main ingredient is a maximal regularity theory for the Stokes system in moving domains.
The precise statement, which we believe to be of independent interest, is given in Theorem \ref{thm:stokesunsteadymoving}.  
Eventually, we turn to the partial regularity proof, which is based on the classical blow-up technique. As in  \cite{SeShSo} and \cite{Br2} we transform the problem locally around a boundary point to a flat geometry obtaining a perturbed system, where the coefficients now also depend on time.
 This requires to re-parametrise the domain by local charts in order to obtain a small Sobolev multiplier norm (uniformly in time) of the coefficients, cf. Section \ref{sec:pert}. 
Our main results in Theorems \ref{thm:main} and \ref{thm:main} state that the velocity field is regular in a.a. boundary point of the moving domain. For this we require that the function $\eta$ is continuous in time with values in the fractional Sobolev space from \eqref{ass:main}
and satisfies additionally
\begin{align}\label{eq:assdteta}
\partial_t\eta\in L^{3}(\mathcal I,W^{1,q_0}(\omega)),\quad q_0>2.
\end{align}
The assumption on the spatial regularity is the natural generalisation of that from \cite{Br2} for fixed domains.
The reason for assumption \eqref{eq:assdteta} is the introduction of local coordinates and, in particular, the rotation of the coordinate system, see Section \ref{sec:repara} for details. This rotation is based
on the normal of the moving domain  (which behaves as $\nabla_y\eta$) averaged in a small spatial region.   Controlling its time-derivative, which naturally enters the picture,  the latter requires \eqref{eq:assdteta}. This assumption is crucial, in various estimates in the proof of Theorem \ref{thm:stokesunsteadymoving}. For the partial regularity result itself we can slightly weaken \eqref{eq:assdteta} to $\partial_t\eta\in L^{3}(\mathcal I; L^{3}\cap W^{1,1}(\omega))$. However, in that the case the existence of suitable weak solutions remains unclear, cf. Remark \ref{rem:assdteta}.

 Below we comment on how \eqref{eq:assdteta}
relates to problems from fluid-structure interaction.

\textbf{Fluid-structure interaction}.
In fluid-structure interaction the function $\eta$ is the displacement of an elastic structure which occupies (a part of) the boundary of the fluid domain.
It is the solution to an evolutionary PDE in its own right and hence of limited regularity. 
A proto-typical example is
\begin{align}\label{eq:eta}
\varrho\partial_t^2\eta-\alpha\partial_t\Delta\eta+\beta\Delta^2\eta=\bfF\quad\text{in}\quad \mathcal I\times \omega
\end{align}
with $\varrho,\alpha,\beta>0$.
The function $\bfF$ on the right-hand side describes the response of the structure to the surface forces of the fluid imposed by the Cauchy stress. Weak solutions to \eqref{eq:eta} belong to the class
\begin{align}\label{eq:2012}
W^{1,\infty}(\mathcal I;L^2(\omega))\cap W^{1,2}(\mathcal I;W^{1,2}(\omega))\cap L^\infty(\mathcal I;W^{2,2}(\omega)),
\end{align}
cf. \cite{LeRu}.
Note that the estimate in the second function space is a consequence of the dissipation in \eqref{eq:eta} (that is, $\alpha>0$) and is not available in the purely elastic case studied in \cite{LeRu}.
Arguing as in
\cite{MuSc} and using the regularising effect from the dissipation one can prove that solutions satisfy additionally
\begin{align}\label{eq:1401b}
\eta\in L^2(\mathcal I;W^{3,2}(\omega)).
\end{align}Due to the compact embeddings
\begin{align}\label{eq:0712}
W^{3,2}(\omega)\hookrightarrow\hookrightarrow W^{1,\infty}(\omega),\quad W^{3,2}(\omega)\hookrightarrow\hookrightarrow W^{2-1/p,p}(\omega),
\end{align}
for some $p>\tfrac{15}{4}$ we obtain $\mathrm{Lip}(\partial\Omega_{\eta(t)})\ll 1$ and $\partial\Omega_{\eta(t)}\in W^{2-1/p,p}(\R^2)$.
But, unfortunately, $\eta$ does not belong to these spaces uniformly in time as required in Theorems \ref{thm:main} and \ref{thm:main'}. Also, the temporal regularity in \eqref{eq:2012} is slightly below that in \eqref{eq:assdteta}
required for the partial regularity theory. In conclusion, the regularity of a weak solution to \eqref{eq:eta} is somewhat below what is required for the partial regularity of solutions to \eqref{1}. 

In fluid structure-interaction one typically has a coupling condition between the velocity field and the shell deformation at the moving boundary and hence non-trivial boundary conditions (which depend on time) for the momentum equation. We do not investigate the effect of irregular (time-dependent) boundary data in this paper, which would require a non-trivial extension of our theory.

\section{Preliminaries and results}
\subsection{Conventions}
We write $f\lesssim g$ for two non-negative quantities $f$ and $g$ if there is a $c>0$ such that $f\leq\,c g$. Here $c$ is a generic constant which does not depend on the crucial quantities and can change from line to line. If necessary we specify particular dependencies. We write $f\approx g$ if $f\lesssim g$ and $g\lesssim f$.
We do not distinguish in the notation for the function spaces between scalar- and vector-valued functions. However, vector-valued functions will usually be denoted in bold case.

\subsection{Classical function spaces}
Let $\mathcal O\subset\R^m$, $m\geq 1$, be open.
Function spaces of continuous or $\alpha$-H\"older continuous functions, $\alpha\in(0,1)$,
 are denoted by $C(\overline{\mathcal O})$ or $C^{0,\alpha}(\overline{\mathcal O})$ respectively. Similarly, we write $C^1(\overline{\mathcal O})$ and $C^{1,\alpha}(\overline{\mathcal O})$.
We denote as usual by $L^p(\mathcal O)$ and $W^{k,p}(\mathcal O)$ for $p\in[1,\infty]$ and $k\in\mathbb N$ Lebesgue- and Sobolev spaces over $\mathcal O$. For a bounded domain $\mathcal O$ the space $L^p_\perp(\mathcal O)$ denotes the subspace of  $L^p(\mathcal O)$ of functions with zero mean, that is $(f)_{\mathcal O}:=\dashint_{\mathcal O}f\dx:=\mathcal L^m(\mathcal O)^{-1}\int_{\mathcal O}f\dx=0$.
 We denote by $W^{k,p}_0(\mathcal O)$ the closure of the smooth and compactly supported functions in $W^{k,p}(\mathcal O)$. This coincides with the functions vanishing $\mathcal H^{m-1}$ -a.e. on $\partial\mathcal O$ provided $\partial\mathcal O$ is sufficiently regular. 
 We also denote by $W^{-k,p}(\mathcal O)$ the dual of $W^{k,p}_0(\mathcal O)$.
  Finally, we consider the subspace
$W^{1,p}_{0,\Div}(\mathcal O)$ of divergence-free vector fields which is defined accordingly. 
We will use the shorthand notations $L^p_x$ and $W^{k,p}_x$ in the case of $3$-dimensional domains and   
$L^p_y$ and $W^{k,p}_y$ for $2$-dimensional sets.

For a separable Banach space $(X,\|\cdot\|_X)$ we denote by $L^p(0,T;X)$ the set of (Bochner-) measurable functions $u:(0,T)\rightarrow X$ such that the mapping $t\mapsto \|u(t)\|_{X}\in L^p(0,T)$. 
The set $C([0,T];X)$ denotes the space of functions $u:[0,T]\rightarrow X$ which are continuous with respect to the norm topology on $(X,\|\cdot\|_X)$.
 The space $W^{1,p}(0,T;X)$ consists of those functions from $L^p(0,T;X)$ for which the distributional time derivative belongs to $L^p(0,T;X)$ as well. 
We use the shorthand $L^p_tX$ for $L^p(0,T;X)$. For instance, we write $L^p_tW^{1,p}_x$ for $L^p(0,T;W^{1,p}(\mathcal O))$. Similarly, $W^{k,p}_tX$ stands for $W^{k,p}(0,T;X)$.

The space $C^{\alpha,\beta}([0,T]\times \overline{\mathcal O})$ with $\alpha,\beta\in(0,1]$ denotes the set of functions being $\alpha$-H\"older continuous
in $t\in[0,T]$ and $\beta$-H\"older continuous in $x\in\overline{\mathcal O}$.

\subsection{Fractional differentiability and Sobolev mulitpliers}
\label{sec:SM}
For $p\in[1,\infty)$ the fractional Sobolev space (Sobolev-Slobodeckij space) with differentiability $s>0$ with $s\notin\mathbb N$ will be denoted by $W^{s,p}(\mathcal O)$. For $s>0$ we write $s=\lfloor s\rfloor+\lbrace s\rbrace$ with $\lfloor s\rfloor\in\N_0$ and $\lbrace s\rbrace\in(0,1)$.
 We denote by $W^{s,p}_0(\mathcal O)$ the closure of the smooth and compactly supported functions in $W^{s,p}(\mathcal O)$. For $s>\frac{1}{p}$ this coincides with the functions vanishing $\mathcal H^{m-1}$ -a.e. on $\partial\mathcal O$ provided $\partial\mathcal O$ is regular enough. We also denote by $W^{-s,p}(\mathcal O)$ for $s>0$ the dual of $W^{s,p}_0(\mathcal O)$. Similar to the case of unbroken differentiabilities above we use the shorthand notations $W^{s,p}_x$  and $W^{s,p}_y$.
We will denote by $B^s_{p,q}(\R^m)$ the standard Besov spaces on $\R^m$ with differentiability $s>0$, integrability $p\in[1,\infty]$ and fine index $q\in[1,\infty]$. They can be defined (for instance) via Littlewood-Paley decomposition leading to the norm $\|\cdot\|_{B^s_{p,q}(\R^m)}$. 
 We refer to \cite{RuSi} and \cite{Tr,Tr2} for an extensive picture. 
 The Besov spaces $B^s_{p,q}(\mathcal O)$ for a bounded domain $\mathcal O\subset\R^m$ are defined as the restriction of functions from $B^s_{p,q}(\R^m)$, that is
 \begin{align*}
 B^s_{p,q}(\mathcal O)&:=\{f|_{\mathcal O}:\,f\in B^s_{p,q}(\R^m)\},\\
 \|g\|_{B^s_{p,q}(\mathcal O)}&:=\inf\{ \|f\|_{B^s_{p,q}(\R^m)}:\,f|_{\mathcal O}=g\}.
 \end{align*}
 If $s\notin\mathbb N$ and $p\in(1,\infty)$ we have $B^s_{p,p}(\mathcal O)=W^{s,p}(\mathcal O)$.
 
In accordance with \cite[Chapter 14]{MaSh} the Sobolev multiplier norm  is given by
\begin{align}\label{eq:SoMo}
\|\varphi\|_{\mathcal M^{s,p}(\mathcal O)}:=\sup_{\bfv:\,\|\bfv\|_{W^{s-1,p}(\mathcal O)}=1}\|\nabla\varphi\cdot\bfv\|_{W^{s-1,p}(\mathcal O)},
\end{align}
where $p\in[1,\infty]$ and $s\geq1$.
The space $\mathcal M^{s,p}(\mathcal O)$ of Sobolev multipliers is defined as those objects for which the $\mathcal M^{s,p}(\mathcal O)$-norm is finite. For $\delta>0$ we denote by  $\mathcal M^{s,p}(\mathcal O)(\delta)$ the subset of functions from
 $\mathcal M^{s,p}(\mathcal O)$ with $\mathcal M^{s,p}(\mathcal O)$-norm not exceeding $\delta$.
By mathematical induction with respect to $s$ one can prove for Lipschitz-continuous functions $\varphi$ that membership to $\mathcal M^{s,p}(\mathcal O)$  in the sense of \eqref{eq:SoMo} implies that
\begin{align}\label{eq:SoMo'}
\sup_{w:\,\|w\|_{W^{s,p}(\mathcal O)}=1}\|\varphi \,w\|_{W^{s,p}(\mathcal O)}<\infty.
\end{align}
The quantity \eqref{eq:SoMo'} also serves as customary definition of the Sobolev multiplier norm in the literature but \eqref{eq:SoMo} is more suitable for our purposes.

Let us finally collect some some useful properties of Sobolev multipliers.
By \cite[Corollary 14.6.2]{MaSh} we have
\begin{align}\label{eq:MSa}
\|\phi\|_{\mathcal M^{s,p}(\R^{m})}\lesssim\|\nabla\phi\|_{L^{\infty}(\R^m)},
\end{align}
provided that one of the following conditions holds:
\begin{itemize}
\item $p(s-1)<m$ and $\phi\in B^{s}_{\varrho,p}(\R^{m})$ with $\varrho\in\big[\frac{m}{s-1},\infty\big]$;
\item $p(s-1)=m$ and $\phi\in B^{s}_{\varrho,p}(\R^m)$ with $\varrho\in(p,\infty]$.
\end{itemize}
Note that the hidden constant in \eqref{eq:MSa} depends on the $B^{s}_{\varrho,p}(\R^{m})$-norm of $\phi$.
By \cite[Corollary 4.3.8]{MaSh} it holds
\begin{align}\label{eq:MSb}
\|\phi\|_{\mathcal M^{s,p}(\R^{m})}\approx
\|\nabla\phi\|_{W^{s-1,p}(\R^{m})} 
\end{align}
for $p(s-1)>m$. 
 Finally, we note the following rule about the composition with Sobolev multipliers which is a consequence of \cite[Lemma 9.4.1]{MaSh}. For open sets $\mathcal O_1,\mathcal O_2\subset\R^m$, $u\in W^{s,p}(\mathcal O_2)$ and a Lipschitz continuous function $\bfphi:\mathcal O_1\rightarrow\mathcal O_2$ with $\bfphi\in \mathcal M^{s,p}(\mathcal O_1)$ and Lipschitz continuous inverse $\bfphi^{-1}:\mathcal O_2\rightarrow\mathcal O_1$ we have
\begin{align}\label{lem:9.4.1}
\|u\circ\bfphi\|_{W^{s,p}(\mathcal O_1)}\lesssim \|u\|_{W^{s,p}(\mathcal O_2)}
\end{align}
with constant depending on $\bfphi$. 

\subsection{Function spaces on variable domains}
\label{ssec:geom}
 The spatial reference domain $\Omega$ is assumed to be an open bounded subset of $\mathbb{R}^3$ with smooth boundary and an outer unit normal ${\bfn}$. We assume that
 $\partial\Omega$ can be parametrised by an injective mapping ${\bfvarphi}\in C^k(\omega;\R^3)$ for some sufficiently large $k\in\N$. We suppose for all points $y=(y_1,y_2)\in \omega$ that the pair of vectors  
$\partial_i {\bfvarphi}(y)$, $i=1,2,$ are linearly independent.
 For a point $x$ in the neighborhood
or $\partial\Omega$ we can define the functions $y$ and $s$ by
\begin{align*}
 y(x)=\arg\min_{y\in\omega}|x-\bfvarphi(y)|,\quad s(x)=(x-\bfp(x))\cdot\bfn(y(x)).
 \end{align*}
Here we used the projection $\bfp(x)=\bfvarphi(y(x))$. We define $L>0$ to be the largest number such that $s,y$ and $\bfp$ are well-defined in $S_L$, where
\begin{align}
\label{eq:boundary1}
S_L=\{x\in\R^3:\,\mathrm{dist}(x,\partial\Omega)<L\}.
\end{align}
Due to the smoothness of $\partial\Omega$ for $L$ small enough we have $\abs{s(x)}=\min_{y\in\omega}|x-\bfvarphi(y)|$ for all $x\in S_L$. This implies that $S_L=\{s\bfn(y)+\bfvarphi(y):(s,y)\in (-L,L)\times \omega\}$.
%
For a given function $\eta :\mathcal I \times \omega \rightarrow\R$ we parametrise the deformed boundary by
\begin{align}\label{eq:bfvarphi}
{\bfvarphi}_\eta(t,y)={\bfvarphi}(y) + \eta(t,y){\bfn}(y), \quad \,y \in \omega,\,t\in \overline{\mathcal I}.
\end{align}
By possibly decreasing $L$, one easily deduces from this formula that $\Omega_{\eta}$ does not degenerate, that is
\begin{align}\label{eq:1705}
\partial_1\bfvarphi_\eta\times\partial_2\bfvarphi_\eta(t,y)&>0 \quad  \,y \in \omega,\,t\in \overline{\mathcal I},
\end{align}
provided $\|\eta\|_{L^\infty_{t,y}}<L$ and $\|\nabla\eta\|_{L^\infty_{t,y}}<L$. Here $\bfn_{\eta(t)}$
is the normal of the domain $\Omega_{\eta(t)}$
 defined through
\begin{align}\label{eq:2612}
\partial\Omega_{\eta(t)}=\set{{\bfvarphi}(y) + \eta(t,y){\bfn}(y):y\in \omega}.
\end{align}
With some abuse of notation we define the deformed space-time cylinder $\mathcal I\times\Omega_\eta=\bigcup_{t\in I}\set{t}\times\Omega_{\eta(t)}\subset\R^{1+3}$.
The corresponding function spaces for variable domains are defined as follows.
\begin{definition}{(Function spaces)}
For $\mathcal I=(0,T)$, $T>0$, and $\eta\in C(\overline{\mathcal I}\times\omega)$ with $\|\eta\|_{L^\infty(\mathcal I\times\omega)}< L$ we define for $1\leq p,r\leq\infty$
\begin{align*}
L^p(\mathcal I;L^r(\Omega_\eta))&:=\big\{v\in L^1(\mathcal I\times\Omega_\eta):\,\,v(t,\cdot)\in L^r(\Omega_{\eta(t)})\,\,\text{for a.e. }t,\,\,\|v(t,\cdot)\|_{L^r(\Omega_{\eta(t)})}\in L^p(\mathcal I)\big\},\\
L^p(\mathcal I;W^{1,r}(\Omega_\eta))&:=\big\{v\in L^p(\mathcal I;L^r(\Omega_\eta)):\,\,\nabla v\in L^p(\mathcal I;L^r(\Omega_\eta))\big\}.
\end{align*}
\end{definition}

\section{The Stokes system in time-dependent domains}
 We consider the unsteady Stokes system
\begin{align}\label{eq:Stokesmoving}
\partial_t\bfu=\Delta \bfu-\nabla\pi+\bff,\quad\Div\bfu=0,\quad\bfu|_{\mathcal I\times\partial{\Omega_\eta}}=0,\quad \bfu(0,\cdot)=\bfu_0,
\end{align}
in a moving domain $\mathcal I\times\Omega_\eta$. Here $\eta:\mathcal I\times\omega\rightarrow\R$ is a given function and we refer to Section \ref{ssec:geom} for the definition of $\Omega_\eta$. The existence of a unique weak solution to \eqref{eq:Stokesmoving} 
 follows from \cite[Theorem 5.1]{NRL} in the special case $p=2$. We are interested in the conditions on $\eta$ which allow for a maximal regularity estimate in the $L^r_tL^p_x$-framework. As a first step we are going to re-paremetrise the boundary of $\Omega_\eta$ locally.  

\subsection{Local re-parametrisation of the boundary}
\label{sec:repara}
Given a function $\eta\in C^0(\overline{\mathcal{I}};C^1(\omega))$ we suppose that there
are numbers $L_0\in(0,L)$ with $L$ given in \eqref{eq:boundary1} and $\kappa_0>0$ such that
\begin{align}\label{eq:L0}
\|\eta\|_{L^\infty(\mathcal I\times \omega)}\leq L_0,\quad \inf_{\mathcal I\times\omega}\partial_1\bfvarphi_\eta\times\partial_2\bfvarphi_\eta\geq \kappa_0,
\end{align}
where $\bfvarphi_\eta$ is defined in accordance with \eqref{eq:bfvarphi}.
Given $y_\star\in \omega$ for some $t\in\overline{\mathcal I}$ fixed
(such that $x_\star:=\bfvarphi(y_\star)+\eta(t,y_\star)\bfn(y_\star)\in\partial\Omega_{\eta(t)}$) there is a neighbourhood
$\omega_\star$ of $y_\star$ such that
\begin{align}\label{eq:L0'}
\bigg(\dashint_{\omega_\star}\partial_1\bfvarphi_\eta\dy\bigg)\times\bigg(\dashint_{\omega_\star}\partial_2\bfvarphi_\eta\dy\bigg)\geq \frac{\kappa_0}{2},
\end{align}
The size of $\omega_\star$ only depends on $\kappa_0$ from \eqref{eq:L0} and $\|\nabla\eta(t)\|_{L^\infty_y}$ and is, in particular, independent of $y_0$. Now \eqref{eq:L0'} implies that we can rotate the coordinate system such that $\dashint_{\omega_\star}\partial_1\bfvarphi_{\eta(t)}\dy$ and  $\dashint_{\omega_\star}\partial_2\bfvarphi_{\eta(t)}\dy$
lie in the $x-y$ plane and $\dashint_{\omega_\star}\bfn_{\eta(t)}\dy$ is orthogonal to it. 
This can be done with the help of an affine linear mapping
\begin{align}\label{eq:V}
\mathscr V_{t,\star}^\eta z=\mathcal Q _{t,\star}^\eta z+x_\star.
\end{align}
Here $\mathcal Q _{t,\star}^\eta\in\R^{3\times 3}$ is an orthogonal matrix which satisfies $\mathcal Q _{t,\star}^\eta(0,0,1)^\top=\dashint_{\omega_\star}\bfn_{\eta(t)}\dy$.
It can be computed explicitly as a function of $\dashint_{\omega_\star}\bfn_{\eta}\dy$ and one easily sees that
\begin{align}\label{eq:dtQ}
|\partial_t \mathcal Q _{t,\star}^\eta|\lesssim\dashint_{\omega_\star}|\partial_t\nabla \eta(t)|\dy\lesssim\int_{\omega}|\partial_t\nabla \eta(t)|\dy,
\end{align} 
where the hidden constant depends on the $L^\infty_tW^{1,\infty}_y$-norm of $\eta$ and assumption \eqref{eq:L0}. Note that the last hidden constant also depends on the size of $\omega_\star$.
We describe now
the boundary locally around $x_\ast$ by the function
$\bfUpsilon_\eta:=(\mathscr V_{t,\star}^\eta)^{-1}\bfvarphi_\eta$.
 Accordingly, it holds
\begin{align}\label{eq:detUps}\mathrm{det}\big(\nabla_y\widetilde\bfUpsilon_{\eta(t)}\big)>0,\quad\widetilde\bfUpsilon_\eta=\begin{pmatrix}\Upsilon^1_\eta\\\Upsilon^2_\eta\end{pmatrix},
\end{align}
in $\omega_\star$. Hence there is
 a neighbourhood $\mathcal O_\star$ of the origin in $\R^2$ such that the function $\widetilde\bfUpsilon_{\eta(t)}:\omega_\star\rightarrow \mathcal O_\star$ is invertible
(using \eqref{eq:L0} one can easily show that $\widetilde\bfUpsilon_{\eta(t)}$ is, in fact, globally invertible in $\omega_\star$). In $\mathcal O_\star$ we define the function
\begin{align}\label{eq:phi}
\bfphi(t,z)
=\begin{pmatrix}z\\ \phi(t,z))\end{pmatrix}
=\begin{pmatrix}z\\ \Upsilon^3_{\eta(t)}((\widetilde\bfUpsilon_{\eta(t)})^{-1}(z))\end{pmatrix}.
\end{align}
It describes the boundary $(\mathscr V^\eta_{t,\star})^{-1}\partial{\Omega}_{\eta(t)}$ close to $0=(\mathscr V^\eta_{t,\star})^{-1}x_\star$ given by $(\mathscr V^\eta_{t,\star})^{-1}\bfvarphi_{\eta(t)}(\omega_\star)$.
One easily checks from the choice of $\mathcal Q_{t,\ast}^\eta$ that $\dashint_{\omega_\star}\nabla_z\Upsilon_\eta^3\,\mathrm d y=0$.
Consequently, $\nabla_z\phi$ is small in $\mathcal O_\star$ provided $\eta(t)\in C^1(\omega)$ and $\mathcal O_\star$ is sufficiently small. 

Suppose now that we have $\eta\in C(\overline{\mathcal I};B^{\theta}_{\varrho,p}\cap C^1(\omega))$, $\theta>2-1/p$, where $p$ and $\varrho$ are related through \begin{align*}
\varrho\geq p\quad\text{if}\quad p\geq 3,\quad \varrho\geq \tfrac{2p}{p-1}\quad\text{if}\quad p< 3,
\end{align*} and that $\sup_I \mathrm{Lip}(\partial\Omega_{\eta(t)})$ is sufficiently small. Then we conclude from \eqref{eq:MSa}, \eqref{eq:MSb} and \eqref{lem:9.4.1} uniformly in time
\begin{align}\label{eq:Br0}
\|\phi\|_{\mathcal M^{2-1/p,p}}\leq\delta,\quad \|\phi\|_{W^{1,\infty}_y}\leq \delta,
\end{align}
for some sufficiently small $\delta$ in $\mathcal O_\star$. Finally, due to the formula
\begin{align*}
\partial_t\phi&=\partial_t\Upsilon^3_{\eta(t)}\circ(\widetilde\bfUpsilon_{\eta(t)})^{-1}+\nabla_y\Upsilon^3_{\eta(t)}\circ(\widetilde\bfUpsilon_{\eta(t)})^{-1}\partial_t(\widetilde\bfUpsilon_{\eta(t)})^{-1}\\
&=\partial_t\Upsilon^3_{\eta(t)}\circ(\widetilde\bfUpsilon_{\eta(t)})^{-1}-\nabla_y\Upsilon^3_{\eta(t)}\circ(\widetilde\bfUpsilon_{\eta(t)})^{-1}\big(\nabla_y\widetilde\bfUpsilon_{\eta(t)}\circ(\widetilde\bfUpsilon_{\eta(t)})^{-1}\big)^{-1}\partial_t\widetilde\bfUpsilon_{\eta(t)}\circ (\widetilde\bfUpsilon_{\eta(t)})^{-1},
\end{align*}
\eqref{eq:bfvarphi}, \eqref{eq:dtQ} and \eqref{eq:detUps} one checks that
\begin{align}\label{eq:dtphi}
|\partial_t\phi\circ\widetilde\bfUpsilon_{\eta(t)}|\lesssim 1+|\partial_t\eta|+\int_\omega|\partial_t\nabla\eta|\dy,
\end{align}
where the hidden constant depends on $L^\infty_tW^{1,\infty}_y$-norm of $\eta$ and the size of $\omega_\star$. Similarly, using the formula
\begin{align*}
\partial_t\nabla_y\phi&=\partial_t\nabla_y\Upsilon^3_{\eta(t)}\circ(\widetilde\bfUpsilon_{\eta(t)})^{-1}\nabla(\widetilde\bfUpsilon_{\eta(t)})^{-1}+\nabla_y^2\Upsilon^3_{\eta(t)}\circ(\widetilde\bfUpsilon_{\eta(t)})^{-1}\nabla_y(\widetilde\bfUpsilon_{\eta(t)})^{-1}\partial_t(\widetilde\bfUpsilon_{\eta(t)})^{-1}\\&+\nabla_y\Upsilon^3_{\eta(t)}\circ(\widetilde\bfUpsilon_{\eta(t)})^{-1}\partial_t\nabla_y(\widetilde\bfUpsilon_{\eta(t)})^{-1},
\end{align*}
we obtain
\begin{align}\label{eq:dtDphi}
|\partial_t\nabla_y\phi\circ\widetilde\bfUpsilon_{\eta(t)}|\lesssim |\partial_t\nabla\eta|+(|\nabla^2\eta|+1)\bigg(1+|\partial_t\eta|+\int_\omega|\partial_t\nabla\eta|\dy\bigg).
\end{align}

 By choosing suitable points $y_1,\dots,y_\ell$ for some $\ell\in\N$ and coordinates as above
we can cover $\partial\Omega_{\eta(t)}$ by open sets\footnote{In fact, $\mathcal U_j$ covers $\bfvarphi_{\eta(t)}(\omega_j)$ for some $\omega_j\subset\omega$, where $\omega_j$ has the role of $\omega_\star$ as in \eqref{eq:L0'}.} $\mathcal U^1,\dots,\mathcal U^\ell$ 
such that
the following holds. 
Setting
$\mathscr V_j:=\mathscr V^\eta_{t,y_j}$
and defining $\phi_j:\mathbb R^{2}\rightarrow\mathbb R$ in accordance with \eqref{eq:phi} there is
$r_j>0$
with the following properties:
\begin{enumerate}[label={\bf (A\arabic{*})}]
\item\label{A1} There is $h_j>0$ such that
$$\mathcal U^j=\{x=\mathscr V_jz\in\mathbb R^3:\,z=(z',z_3)\in\R^3,\,|z'|<r_j,\,
|z_3-\phi_j(z')|<h_j\}.$$
\item\label{A2} For $x\in\mathcal U^j$ we have with $z=\mathscr V_j^{-1}(x)$
\begin{itemize}
\item $x\in\partial\Omega_{\eta(t)}$ if and only if $z_3=\phi_j(z')$;
\item $x\in\Omega_{\eta(t)}$ if and only if $0<z_3-\phi_j(z')<h_j$;
\item $x\notin\Omega_{\eta(t)}$ if and only if $0>z_3-\phi_j(z')>-h_j$.
\end{itemize}
\item\label{A3} We have that
$$\partial\Omega_{\eta(t)}\subset \bigcup_{j=1}^\ell\mathcal U^j.$$
\end{enumerate}
If $\eta\in C(\overline{\mathcal I};C^{1}(\omega))$ and
\ref{A1}--\ref{A3} hold for some $t\in\overline{\mathcal I}$
we can slightly vary the value of $t$ such that \ref{A1}--\ref{A3} continue to hold with the same choices of $y_j$, $r_j$ and $h_j$, $j=1,\dots,\ell$. Similarly, it is not necessary to change $\mathcal U^j$.

We need to extend the functions $\varphi_1,\dots,\varphi_\ell$  from \ref{A1}--\ref{A3} to the half space
$\mathbb H := \set{\xi = (\xi',\xi_3)\,:\, \xi_3 > 0}$. Hence we are confronted with the task of extending a function~$\phi\,:\, \R^{2}\to \R$ to a mapping $\bfPhi\,:\, \mathbb H \to \R^3$.
This can be done using the extension operator of Maz'ya and Shaposhnikova~\cite[Section 9.4.3]{MaSh}. Let $\zeta \in C^\infty_c(B_1(0'))$ with $\zeta \geq 0$ and $\int_{\R^{2}} \zeta(x')\dx'=1$. Let $\zeta_t(x') := t^{-2} \zeta(x'/t)$ denote the induced family of mollifiers. We define the extension operator 
\begin{align*}
  (\mathcal{T}\phi)(\xi',\xi_3):=\int_{\R^{2}} \zeta_{\xi_3}(\xi'-y')\phi(y')\dy',\quad (\xi',\xi_3) \in \mathbb H,
\end{align*}
where~$\phi:\R^2\to \R$ is a Lipschitz function with Lipschitz constant~$L$.
Then the estimate 
\begin{align}\label{est:ext}
  \norm{\nabla (\mathcal{T} \phi)}_{W^{s,p}(\setR^{3})}\le c\norm{\nabla \phi}_{W^{s-\frac 1 p,p}(\setR^{2})}
\end{align}
follows from~\cite[Theorem 8.7.2]{MaSh}. Moreover, \cite[Theorem 8.7.1]{MaSh} yields
\begin{align}\label{eq:MS}
\|\mathcal T\phi\|_{\mathcal M^{s,p}(\mathbb H)}\lesssim \|\phi\|_{\mathcal M^{s-\frac{1}{p},p}(\R^{2})}.
\end{align}
It is shown in \cite[Lemma 9.4.5]{MaSh} that (for sufficiently large~$N$, i.e., $N \geq c(\zeta) L+1$) the mapping
\begin{align*}
  \alpha_{z'}(z_3) \mapsto N\,z_3+(\mathcal{T} \phi)(z',z_3)
\end{align*}
is for every $z' \in \setR^{2}$ one to one and the inverse is Lipschitz with gradient
bounded by $(N-L)^{-1}$.
Now, we define the mapping~$\bfPhi\,:\, \mathbb H \to \R^3$ as a rescaled version of the latter one by setting
\begin{align}\label{eq:Phi}
  \bfPhi(\xi',\xi_3)
  &:=
    \big(\xi',
    \alpha_{\xi_3}(\xi')\big) = 
    \big(\xi',
    \,\xi_3 + (\mathcal{T} \phi)(\xi',\xi_3/N)\big).
\end{align}
Thus, $\bfPhi$ is one-to-one (for sufficiently large~$N=N(L)$) and we can define its inverse $\bfPsi := \bfPhi^{-1}$.
 The Jacobi matrix of the mapping $\bfPhi$ satisfies
\begin{align}\label{J}
  J = \nabla \bfPhi = 
  \begin{pmatrix}
    \mathbb I_{2\times 2}&0
    \\
    \partial_{\xi'} (\mathcal{T}  \phi)& 1+ 1/N\partial_{\xi_3}(\mathcal{T}  \phi)
  \end{pmatrix}.
\end{align}
Since 
$\abs{\partial_{\xi_3}\mathcal{T}  \phi} \leq L$, we have \begin{align}\label{eq:detJ}\frac{1}{2} < 1-L/N \leq \abs{\det(J)} \leq 1+L/N\leq 2\end{align}
using that $N$ is large compared to~$L$. Finally, we note the implication
\begin{align} \label{eq:SMPhiPsi}
\bfPhi\in\mathcal M^{s,p}(\mathbb H)\,\,\Rightarrow \,\,\bfPsi\in \mathcal M^{s,p}(\mathbb H),
\end{align}
which holds, for instance, if $\bfPhi$ is Lipschitz continuous, cf. \cite[Lemma 9.4.2]{MaSh}.

\subsection{Maximal regularity theory}
\label{sec:stokesunsteadymoving}
With the preparations from the previous subsection at hand we are now able to prove the following maximal regularity theorem for the Stokes system in moving domains.
\begin{theorem}\label{thm:stokesunsteadymoving}
Let $p,r\in(1,\infty)$ and 
\begin{align}\label{eq:SMp''}
\varrho\geq p\quad\text{if}\quad p\geq 3,\quad \varrho\geq \tfrac{2p}{p-1}\quad\text{if}\quad p< 3.
\end{align}
 Suppose that $\eta\in C(\overline{\mathcal I};B^{\theta}_{\varrho,p}\cap W^{2,2}\cap C^{1}(\omega))$ for some $\theta>2-1/p$, that $\sup_{\mathcal I} \mathrm{Lip}(\partial\Omega_{\eta(t)})$ is sufficiently small and that \eqref{eq:L0} holds. Suppose further that $\partial_t\eta\in L^{r_0}(\mathcal I; W^{1,q_0}(\omega))$ for some $r_0>\max\{2,r,\frac{2p'}{3}\}$ and $q_0>2$, $\bff\in L^r(\mathcal I;L^{p}(\Omega_\eta))$ and $\bfu_{0}\in W^{2,p}\cap W^{1,p}_{0,\Div}(\Omega_{\eta_0})$. 
Then there is a unique solution to \eqref{eq:Stokesmoving} and we have
\begin{align}\label{eq:mainparacor}
\|\partial_t&\bfu\|_{L^r(\mathcal I;L^{p}(\Omega_\eta))}+\|\bfu\|_{L^r(\mathcal I;W^{2,p}(\Omega_\eta))}
+\|\pi\|_{L^r(\mathcal I;W^{1,p}(\Omega_\eta))}\lesssim\|\bff\|_{L^r(\mathcal I;L^{p}(\Omega_\eta))}+\|\bfu_{0}\|_{W^{2,p}(\Omega_{\eta_0})}.
\end{align}
\end{theorem}
\begin{remark}
The only comparable result concerning the maximal regularity result for the Stokes problem in non-cylindrical domains is given in \cite[Section 2.10]{HS}. The authors consider there smooth domains (of class $C^3$) and suppose that the transformation of the domain induced by $\eta$ is volume preserving. This is restrictive for applications in fluid-structure interaction.
\end{remark}
\begin{remark}\label{rem:lipsmall}
It holds $\sup_I \mathrm{Lip}(\partial\Omega_{\eta(t)})\ll 1$ for instance if $t\mapsto \eta(t)$ maps boundedly into a function space
strictly smaller than $W^{1,\infty}(\omega)$ (such as $C^{1,\alpha}(\omega)$ for some $\alpha>0$) or even if $\eta\in C^0(\overline{\mathcal{I}},C^1(\omega))$. In these cases one can re-parametrise the boundary by local charts (as done in Section \ref{sec:repara} above)
obtaining a small local Lipschitz constant $\mathrm{Lip}(\partial\Omega_{\eta(t)})$.
\end{remark}
\begin{proof}[Proof of Theorem \ref{thm:stokesunsteadymoving}.]
Suppose that $\bfu_0=0$ (otherwise, one can consider $\bfv:=\bfu-\bfu_0$).
For $t\in\overline{\mathcal I}=[0,T]$ with $T\ll1$ we consider functions $\phi_1,\dots,\phi_\ell$ satisfying \ref{A1}--\ref{A3} as introduced in Section \ref{sec:repara}, cf. equation \eqref{eq:phi}.
We clearly find an open set $\mathcal U^0$ such that ${\Omega_{\eta(t)}}\subset \cup_{j=0}^\ell \mathcal U^j$ for all $t\in\overline{\mathcal I}$.
 Finally, we consider a decomposition of unity $(\xi_j)_{j=0}^\ell$ with respect to the covering
$\mathcal U^0,\dots,\mathcal U^\ell$. 
For $j\in\{1,\dots,\ell\}$ we consider the extension $\bfPhi_j$ of $\phi_j$ given by \eqref{est:ext} with inverse $\bfPsi_j$. Note that $\phi_j,\bfPhi_j$ and $\bfPsi_j$ also depend on $t$, whereas $\xi_j$ does not.
We define the operators\footnote{Since $\partial\Omega_\eta$ is Lipschitz uniformly in time we can use a standard extension operator to extend functions in \eqref{eq:operators}--\eqref{eq:stokeshalf} to the whole space or half space when necessary.}
\begin{align}\label{eq:operators}
\mathscr R_0\bff&:=\xi_0 \bfU_0+\sum_{j=1}^\ell\xi_j \bfU_j \circ\bfPsi_j\circ\mathscr V_j ,\quad
\mathscr P\bff:=\sum_{j=1}^\ell\xi_j\mathfrak q_j \circ\bfPsi_j\circ\mathscr V_j,
\end{align}
\begin{align*}
\mathscr R\bff=\mathscr R_0\bff+\mathscr R_1\bff,\quad \mathscr R_1\bff=-\Bog_{{\Omega_{\eta(t)}}}\Div\mathscr R_0\bff,
\end{align*}
with the Bogovskii-operator $\Bog_{\Omega_{\eta(t)}}$ on the moving domain. The latter has been extensively analysed in \cite{SaSc} and it is shown that it has the expected properties due to our assumption $\eta\in C^{0}(\overline{\mathcal I};C^1(\omega))$. In particular, it holds
\begin{align}\label{Bog}
\partial_t\mathrm{Bog}_{\Omega_\eta}\Div:W^{1,p}(\mathcal I;L^{p}(\Omega_\eta))
\rightarrow L^p(\mathcal I;L^{p}(\Omega_\eta)).
\end{align}
The functions $(\bfU_0,\mathfrak q_0)$ and $(\bfU_j,\mathfrak q_j)$ for $j\in\{1,\dots,\ell\}$ are the solutions to the Stokes problem on the whole space and the half space respectively with data $\bff$ (transformed if necessary), that is, we have
\begin{align}\label{eq:stokeswhole}
\partial_t\bfU_0=\Delta \bfU_0-\nabla\mathfrak q_0+\bff,\quad\Div\bfU_0=0,\quad \bfU_0(0,\cdot)=0,
\end{align}
and 
\begin{align}\label{eq:stokeshalf}
\partial_t\bfU_j=\Delta \bfU_j-\nabla\mathfrak q_j+\bff\circ\mathscr V_j^{-1}\circ \bfPhi_j,\quad\Div\bfU_j=0,\quad\bfU_j|_{I\times\partial\mathbb H}=0,\quad \bfU_j(0,\cdot)=0.
\end{align}
We have 
\begin{align}\label{est:stokeswhole}
\begin{aligned}
\int_{\mathcal I}\Big(\|\partial_t\bfU_0\|_{L^{p}_x}^r+\|\nabla^2\bfU_0\|^r_{L^{p}_x}+\|\nabla\mathfrak q_0\|_{L^{p}_x}^r\Big)\dt&\lesssim \int_{\mathcal I}\|\bff\|_{L^{p}_x}^r\dt,\\
\int_{\mathcal I}\Big(\|\partial_t\bfU_0\|_{W^{-1,p}_x}^r+\|\bfU_0\|_{W^{1,p}_x}^r+\|\mathfrak q_0\|_{L^{p}_x}^r\Big)\dt&\lesssim T^{r/2}\int_{\mathcal I}\|\bff\|_{L^{p}_x}^r\dt,
\end{aligned}
\end{align}
and for $j=1,\dots,\ell$ (using Lipschitz continuity of $\bfPhi_j$ which follows from that of $\phi_j$, cf. \eqref{J}, and that $\mathcal Q_j$ is an orthogonal matrix)
\begin{align}\label{est:stokeshalf}
\begin{aligned}
\int_{\mathcal I}\Big(\|\partial_t\bfU_j\|_{L^{p}_x}^r+\|\nabla^2\bfU_j\|_{L^{p}_x}^r+\|\nabla\mathfrak q_j\|_{L^{r}_x}^r\Big)\dt&\lesssim\int_{\mathcal I}\|\bff\circ\mathscr V_j\circ \bfPhi_j\|_{L^{p}_x}^r\dt\lesssim \int_{\mathcal I}\|\bff\|_{L^{p}_x}^r\dt,\\
\int_{\mathcal I}\Big(\|\partial_t\bfU_j\|_{W^{-1,p}_x}^r+\|\bfU_j\|_{W^{1,p}_x}^r+\|\mathfrak q_j\|_{L^{p}_x}^r\Big)\dt&\lesssim T^{r/2}\int_{\mathcal I}\|\bff\circ\mathscr V_j\circ \bfPhi_j\|_{L^{p}_x}^r\dt\lesssim T^{r/2}\int_{\mathcal I}\|\bff\|_{L^{p}_x}^r\dt,
\end{aligned}
\end{align}
uniformly in $T$. Note that estimates \eqref{est:stokeswhole}$_2$ and \eqref{est:stokeshalf}$_2$ only hold locally in space (that is, in balls $B\subset\R^n$ with a constant depending on the radius).  
Estimates \eqref{est:stokeswhole} and \eqref{est:stokeshalf} are classical in the case $r=p$, see \cite[Theorems 3.1 \& 3.2]{So}.
For the case of arbitrary exponents $p$ and $r$ we refer to \cite{HS} and the references therein.
The $T$-dependence in \eqref{est:stokeswhole}$_2$ and \eqref{est:stokeshalf}$_2$ follows by scaling.

Setting $\bfV_j=\bfU_j\circ\bfPsi_j\circ\mathscr V_j$ and $\mathfrak Q_j=\mathfrak q_j\circ\bfPsi_j\circ\mathscr V_j $, we obtain
\begin{align}\label{eq:Stokesback}
\begin{aligned}
\partial_t\bfV_j=&\Delta\bfV_j-\nabla\mathfrak Q_j+(1-\mathrm{det}(\nabla\bfPsi_j))\partial_t\bfV_j-\Div\big((\mathbb I_{3\times 3}-\bfA_j)\nabla\bfV_j)-\Div((\mathbf{B}_j-\mathbb I_{3\times 3})\mathfrak Q_j)\\&-\mathrm{det}(\nabla\bfPsi_j)\nabla\bfV_j(\partial_t\mathscr V_j^{-1}\circ\mathscr V_j+\mathcal Q_j^\top\partial_t\bfPhi_j\circ \bfPsi_j\circ\mathscr V_j)+\bff,\\
&\Div\bfV_j=(\mathbb I_{3\times 3}-\mathbf{B}_j)^\top:\nabla\bfV_j,\quad\bfV_j|_{\mathcal I\times\partial{\Omega_\eta}\cap \mathcal U^j}=0,\quad \bfV_j(0,\cdot)=0,
\end{aligned}
\end{align}
where $\bfA_j:=\mathrm{det}(\nabla\bfPsi_j)\nabla\bfPhi_j^\top\circ\bfPsi_j\nabla\bfPhi_j\circ\bfPsi_j$ and $\mathbf{B}_j:=\mathrm{det}(\nabla\bfPsi_j)\mathcal Q_j^\top\nabla\bfPhi_j\circ\bfPsi_j$.
 There holds
\begin{align}
\partial_t\mathscr R\bff&-\Delta\mathscr R\bff+\nabla \mathscr P\bff=\bff+\mathscr S\bff+(\partial_t-\Delta)\mathscr R_1\bff,\label{eq:4.10}\\
\mathscr S\bff&=-\nabla\bfV_0\nabla\xi_0-\Div\big(\nabla\xi_0\otimes\bfV_0\big)
-\sum_{j=1}^\ell\nabla\bfV_j \nabla\xi_j\nonumber\\&-\sum_{j=1}^\ell\Div\big(\nabla\xi_j\otimes\bfV_j\big)+\sum_{j=1}^\ell\nabla\xi_j \mathfrak Q_j
-\sum_{j=1}^\ell\xi_j\Div((\mathbf{B}_j-\mathbb I_{3\times 3})\mathfrak Q_j)\nonumber\\
&-\sum_{j=1}^\ell\xi_j\Div\big((\mathbb I_{3\times 3}-\bfA_j)\nabla\bfV_j)+\sum_{j=1}^\ell\xi_j(1-\mathrm{det}(\nabla\bfPsi_j))\partial_t\bfV_j\nonumber\\
&-\sum_{j=1}^\ell\xi_j\mathrm{det}(\nabla\bfPsi_j)\nabla\bfV_j(\partial_t\mathscr V_j^{-1}\circ\mathscr V_j+\mathcal Q_j^\top\partial_t\bfPhi_j\circ \bfPsi_j\circ\mathscr V_j)=:\sum_{i=1}^9 \mathscr S_i\bff.\label{eq:4.11}
\end{align}
We want to invert the operator
\begin{align*}
\mathscr L&:\mathscr Y_{r,p}^\eta\rightarrow L^r(\mathcal I; L^p_{\Div}(\Omega_\eta)),\quad
\bfv\mapsto\mathcal P_p^\eta\big(\partial_t\bfv-\Delta\bfv\big),
\end{align*}
where the space $\mathscr Y_{r,p}^\eta$ is given by
\begin{align*}
\mathscr Y_{r,p}^\eta:= L^r(\mathcal I;W^{1,p}_{0,\Div}\cap W^{2,p}({\Omega_\eta}))\cap W^{1,r}(\mathcal I;L^{p}({\Omega_\eta}))\cap\set{\bfv:\,\bfv(0,\cdot)=0}.
\end{align*}
From (\ref{eq:4.10}) it follows
\begin{align*}
\mathscr L\mathscr R\bff=\bff+\mathcal P^\eta_p\mathscr S\bff+\mathcal P^\eta_p(\partial_t-\Delta)\mathscr R_1\bff,
\end{align*}
i.e.,
\begin{align}\label{eq:IS}
\mathscr L\circ\mathscr R=\mathrm{id}+\mathscr T,
\end{align}
 with $\mathscr T=\mathcal P_p^\eta\mathscr S+\mathcal P^\eta_p(\partial_t-\Delta)\mathscr R_1$. 
Here $\mathcal P_p^\eta$ is the Helmholtz projection from $L^p({\Omega_\eta})$ onto $L^p_{\Div}({\Omega_\eta})$. 
The Helmholtz-projection $\mathcal P_p \bfu$ of a function $\bfu\in L^p({\mathcal O})$, $\mathcal O\subset\R^3$ bounded with normal $\bfn_{\mathcal O}$, is defined as $\mathcal P_p \bfu:=\bfu-\nabla h$, where $h$ is the solution to the Neumann-problem
\begin{align*}
\begin{cases}\Delta h=\Div \bfu\quad\text{in}\quad {\mathcal O},\\
\bfn_{\mathcal O}\cdot(\nabla h-\bfu)=0\quad\text{on}\quad\partial {\mathcal O}.
\end{cases}
\end{align*}
The aim is now to prove that the operator-norm of $\mathscr T$ from \eqref{eq:IS} is strictly smaller than 1, which implies that $\mathscr L$ is surjective. 

Due to the assumption $\eta\in C^0(\overline{\mathcal I};B^{\theta}_{\varrho,p}(\omega))$ with $\theta>2-1/p$ and \eqref{eq:Br0} the terms $\mathscr S_1\bff,\dots,\mathscr S_8\bff$ can be estimated exactly as in \cite[Proof of Theorem 3.1]{Br2} obtaining
\begin{align}\label{eq:july24a}
\int_{\mathcal I}\|\mathscr S_i\bff\|^r_{L^{p}_x}\dt\leq\,\delta(T)\int_{\mathcal I}\|\bff\|_{L^{p}_x}^r\dt,\quad i=1,\dots,8,
\end{align}
where $\delta(T)\rightarrow0$ as $T\rightarrow0$.
The term $\mathscr S_9\bff$ appears on account of the time-dependence of the boundary and does not have a counterpart in \cite{Br2}. 
In order to esitmate it we note that by \eqref{eq:Br0} and \eqref{eq:MS} we have uniformly in time
 $\bfPhi_j\in \mathcal M^{2,p}(\mathbb H)$ and thus $\bfPsi_j\in \mathcal M^{2,p}(\mathbb H)$ by \eqref{eq:SMPhiPsi}. 
This justifies the estimate 
\begin{align*}
\int_{\mathcal I}\|\bfV_j\|_{W^{2,p}_x}^r\dt\lesssim \int_{\mathcal I}\|\bfU_j\|_{W^{2,p}_x}^r\dt
\end{align*}
due to \eqref{lem:9.4.1}.
Using the embeddings
\begin{align*}
W^{1,r}(\mathcal I,L^p(\Omega))\cap L^r(\mathcal I,W^{2,p})\hookrightarrow W^{\frac{1}{2},r}(\mathcal I,W^{1,p}(\Omega))\hookrightarrow L^{\overline r}(\mathcal I,W^{1,p}(\Omega)),\quad W^{1,q_0}(\omega)\hookrightarrow L^\infty(\omega),
\end{align*}
where $\overline r=\infty$ if $r>2$, $\overline r<\infty$ arbitrary if $r=2$ and $\overline r=\frac{2r}{2-r}$ if $r<2$, as well as \eqref{eq:dtQ} and \eqref{eq:dtphi} (together with \eqref{est:ext} and the definition of $\bfPhi_j$) yields
\begin{align}\label{eq:0301}
\begin{aligned}
\int_{\mathcal I}\|\mathscr S_9\bff\|^r_{L^{p}_x}\dt
&\lesssim \sum_{j=1}^\ell\|\nabla\bfU_j\|_{L^{\overline r}(\mathcal I,L^{p}(\Omega))}^r\big(\|\partial_t\nabla\eta\|_{L^{\tilde r}(\mathcal I,L^{1}(\omega))}^r+\|\partial_t\eta\|_{L^{\tilde r}(\mathcal I,L^{\infty}(\omega))}^r\big)\\
&\lesssim \delta(T)\sum_{j=1}^\ell\|\nabla\bfU_j\|_{L^{\overline r}(\mathcal I,L^p(\Omega))}^r\|\partial_t\eta\|_{L^{r_0}(\mathcal I,W^{1,q_0}(\omega))}^r\\
&\lesssim \delta(T)\sum_{j=1}^\ell\big(\|\partial_t\bfU_j\|_{L^{r}(\mathcal I,L^p(\mathbb H))}^r+\|\bfU_j\|^r_{L^{r}(\mathcal I,W^{2,p}(\mathbb H))}\big)\\
&\lesssim \delta(T)\int_{\mathcal I}\|\bff\|_{L^{p}_x}^r\dt,
\end{aligned}
\end{align}
where $\max \{2,r\}<\tilde r<r_0$. Note that we also used \eqref{est:stokeshalf} in the last step and the assumption $\eta\in L^{r_0}(\mathcal I,W^{1,q_0}(\omega))$ in the second last one.
Combining this with \eqref{eq:july24a} and choosing $T$ small enough we can infer that
\begin{align}\label{eq:july24}
\int_{\mathcal I}\|\mathscr S\bff\|^r_{L^{p}_x}\dt\leq\,\tfrac{1}{4}\int_{\mathcal I}\|\bff\|_{L^{p}_x}^r\dt.
\end{align}
Now we are going to show the same for $(\partial_t-\Delta)\mathscr R_1$. 
 We have since $\Div\bfB_j^\top=0$
\begin{align*}
\Div\mathscr R_0\bff&=\nabla\xi_0\cdot\bfU_0+\sum_{j=1}^\ell\nabla\xi_j\cdot\bfV_j+\sum_{j=1}^\ell\xi_j(\mathbb I_{3\times 3}-\bfB_j)^\top:\nabla\bfV_j\\
&=\nabla\xi_0\cdot\bfU_0+\sum_{j=1}^\ell\nabla\xi_j\cdot\bfV_j+\sum_{j=1}^\ell\xi_j\Div\big((\mathbb I_{3\times 3}-\bfB_j)^\top:\bfV_j\big)
\end{align*}
such that \eqref{Bog} yields
\begin{align*}
\|\partial_t\mathscr R_1\bff\|_{L^{p}_x}
&\lesssim \|\nabla\xi_0\cdot\partial_t\bfU_0\|_{W^{-1,p}_x}+\sum_{j=1}^\ell\|\nabla\xi_j\cdot\partial_t\bfV_j\|_{W^{-1,p}_x}\\
&+ \sum_{j=1}^\ell\|\xi_j\Div\big((\mathbb I_{3\times 3}-\bfB_j)^\top\partial_t\bfV_j\big)\|_{W^{-1,p}_x}\\
&+ \sum_{j=1}^\ell\|\xi_j\partial_t\bfB_j^\top:\nabla\bfV_j\|_{W^{-1,p}_x}\\
&=: (R)_1+(R)_2+(R)_3+(R)_4.
\end{align*}
Since $\bfU_0$ solves \eqref{eq:stokeswhole} we infer from \eqref{est:stokeswhole} that
\begin{align*}
\int_{\mathcal I}(R)_1^r\dt\lesssim \int_{\mathcal I}\|\partial_t\bfU_0\|_{W^{-1,p}_x}^r\dt\lesssim T^{r/2}\int_{\mathcal I}\|\bff\|_{L^{p}_x}^r\dt.
\end{align*}
In order to estimate the remaining terms we need
to estimate $\bfV_j$ in terms of $\bfU_j$. This is more complicated than in \cite{Br2} on account of the time-dependence of the transformation $\bfPsi_j\circ\mathscr V_j$. Since this functions is uniformly Lipschitz (which follows from Lipschitz continuity of $\phi_j$, cf. \eqref{J}, and that the fact that $\mathcal Q_j$ is an orthogonal matrix) we have
\begin{align*}
\|\partial_t\bfV_j\|_{W^{-1,p}_x}&=\|\partial_t(\bfU_j\circ\bfPsi_j\circ\mathscr V_j)\|_{W^{-1,p}_x}\\
&\lesssim \|(\partial_t\bfU_j)\circ\bfPsi_j\circ\mathscr V_j)\|_{W^{-1,p}_x}+ \|\nabla\bfU_j\circ\bfPsi_j\circ\mathscr V_j\partial_t\bfPsi_j\circ\mathscr V_j\|_{W^{-1,p}_x}\\&+\|\nabla\bfU_j\circ\bfPsi_j\circ\mathscr V_j\nabla\bfPsi_j\circ\mathscr V_j\partial_t\mathscr V_j\|_{W^{-1,p}_x}\\
&\lesssim \|\partial_t\bfU_j\|_{W^{-1,p}_x}+ \|\nabla\bfU_j\circ\bfPsi_j\partial_t\bfPsi_j\|_{W^{-1,p}_x}\\&+\|\nabla\bfU_j\circ\bfPsi_j\circ\mathscr V_j\nabla\bfPsi_j\circ\mathscr V_j\partial_t\mathscr V_j\|_{W^{-1,p}_x}.
\end{align*}
Choosing $\tilde r$ and $\overline r$ as in \eqref{eq:0301} and using \eqref{eq:dtphi} (together with \eqref{est:ext} and the definition of $\bfPhi_j$ and $\bfPsi_j$) we have for $\tilde p:=\max\{1,3p/(p+3)\}$
\begin{align*}
\|\nabla\bfU_j\circ\bfPsi_j\partial_t\bfPsi_j\|_{L^r(\mathcal I,W^{-1,p}(\Omega))}&\lesssim \bigg(\int_{\mathcal I}\|\nabla\bfU_j\circ\bfPsi_j\partial_t\bfPsi_j\|_{L^{\tilde p}_x}^r\dt\bigg)^{\frac{1}{r}}\\
&\lesssim \bigg(\int_{\mathcal I}\|\nabla\bfU_j\|_{L^p_x}^r\|\partial_t\bfPsi_j\|_{L^{\max\{3,p'\}}_x}^r\dt\bigg)^{\frac{1}{r}}\\
&\lesssim \bigg(\int_{\mathcal I}\|\nabla\bfU_j\|_{L^p_x}^r\|\partial_t\phi_j\|_{L^{\infty}_y}^r\dt\bigg)^{\frac{1}{r}}\\
&\lesssim \|\nabla\bfU_j\|_{L^{\overline r}(\mathcal I,L^p(\Omega))}\big(\|\partial_t\eta\|_{L^{\tilde r}(\mathcal I,L^{\infty}(\omega))}+\|\partial_t\nabla\eta\|_{L^{\tilde r}(\mathcal I,L^{1}(\omega))}\big)\\
&\lesssim \|\nabla\bfU_j\|_{L^{\overline r}(\mathcal I,L^p(\Omega))}\|\partial_t\eta\|_{L^{\tilde r}(\mathcal I,W^{1,q_0}(\omega))}\\
&\lesssim \delta(T)\|\nabla\bfU_j\|_{L^{\overline r}(\mathcal I,L^p(\Omega))}\|\partial_t\eta\|_{L^{r_0}(\mathcal I,W^{1,q_0}(\omega))}\\
&\lesssim \delta(T)\big(\|\partial_t\bfU_j\|_{L^{r}(\mathcal I,L^p(\Omega))}+\|\bfU_j\|_{L^{r}(\mathcal I,W^{2,p}(\Omega))}\big),
\end{align*}
and, similarly,
\begin{align*}
\|\nabla\bfU_j\circ\bfPsi_j\circ\mathscr V_j\nabla\bfPsi_j\circ\mathscr V_j\partial_t\mathscr V_j\|_{L^r(\mathcal I,W^{-1,p}(\Omega))}
&\lesssim \delta(T)\big(\|\partial_t\bfU_j\|_{L^{r}(\mathcal I,L^p(\Omega))}+\|\bfU_j\|_{L^{r}(\mathcal I,W^{2,p}(\Omega))}\big)
\end{align*}
on account of \eqref{eq:dtQ}.
In conclusion, we have
\begin{align*}
\int_{\mathcal I}\|\partial_t\bfV_j\|_{W^{-1,p}_x}^r\dt\lesssim\delta(T)\big(\|\partial_t\bfU_j\|_{L^{r}(\mathcal I,L^p(\Omega))}^r+\|\bfU_j\|_{L^{r}(\mathcal I,W^{2,p}(\Omega))}^r\big).
\end{align*}
With this at hand we obtain from \eqref{est:stokeshalf}
\begin{align*}
\int_{\mathcal I}(R)_2^r\dt\lesssim \sum_{j=1}^\ell\int_{\mathcal I}\|\partial_t\bfV_j\|_{W^{-1,p}_x}^r\dt\lesssim \delta(T)\int_{\mathcal I}\|\bff\|_{L^{p}_x}^r\dt.
\end{align*}
In the following we apply Necas' negative norm theorem in $\Omega_\eta$, which holds since $\partial\Omega_\eta$ is Lipschitz uniformly in time.
Using that $\bfPsi\circ\mathscr V_j$ is Lipschitz we easily obtain from \eqref{eq:Br0} (argiung as in the estimate for $(R)_2^r$ above)
\begin{align*}
\int_{\mathcal I}(R)_3^r\dt&\lesssim  \sum_{j=1}^\ell\int_{\mathcal I}\|\mathbb I_{3\times 3}-\mathbf{B}_j\|_{ L^{\infty}({\Omega})}^r\|\partial_t\bfV_j\|_{L^{p}_x}^r\dt\\
&\lesssim  \delta(\mathrm{Lip}(\partial{\Omega_\eta}))\int_{\mathcal I}\big(\|\partial_t\bfU_j\|_{L^{p}_x}^r+\|\nabla\bfU_j\partial_t(\bfPsi_j\circ\mathscr V_j)\circ(\bfPsi_j\circ\mathscr V_j)^{-1}\|_{L^{p}_x}^r\big)\dt\\
&\lesssim  \delta(\mathrm{Lip}(\partial{\Omega_\eta}))\int_{\mathcal I}\|\bff\|_{L^{p}_x}^r\dt+\int_{\mathcal I}\|\nabla\bfU_j\|_{L^p_x}^r\big(\|\partial_t\eta\|^r_{L^\infty(\omega)}+\|\partial_t\nabla\eta\|^r_{L^1(\omega)}\big)\dt\\
&\lesssim  \delta(\mathrm{Lip}(\partial{\Omega_\eta}))\int_{\mathcal I}\|\bff\|_{L^{p}_x}^r\dt+\delta(T)\big(\|\partial_t\bfU_j\|_{L^{r}(\mathcal I,L^p(\Omega))}^r+\|\bfU_j\|_{L^{r}(\mathcal I,W^{2,p}(\Omega))}^r\big)\\
&\lesssim \delta(\mathrm{Lip}(\partial{\Omega_\eta}))\int_{\mathcal I}\|\bff\|_{L^{p}_x}^r\dt+\delta(T)\int_{\mathcal I}\|\bff\|_{L^{p}_x}^r\dt
\end{align*}
using again \eqref{est:stokeshalf}. Note carefully that the estimate of the second term  requires to take $T$ much smaller than $ \delta(\mathrm{Lip}(\partial{\Omega_\eta}))$ since estimates \eqref{eq:dtQ} and \eqref{eq:dtphi} depend on the size of $\mathcal U^j$ and thus on the reciprocal of $\mathrm{Lip}(\partial{\Omega_\eta})$.
 As far as $(R)_4$ is concerned let us first suppose that $p\geq\frac{3}{2}$ such that $\frac{3p}{p+3}\geq 1$. In this case we have by \eqref{eq:dtphi} and \eqref{eq:dtDphi}
\begin{align*}
\int_{\mathcal I}(R)_4^r\dt&\lesssim\sum_{j=1}^\ell \int_{\mathcal I}
\|\partial_t\bfB_j^\top:\nabla\bfV_j\|_{L^{3p/(p+3)}_x}^r\dt\lesssim\sum_{j=1}^\ell \int_{\mathcal I}
\|\partial_t\bfB_j\|_{L^3_x}^r\|\nabla\bfV_j\|_{L^{p}_x}^r\dt\\
&\lesssim\sum_{j=1}^\ell \int_{\mathcal I}
\big(\|\partial_t\nabla\bfPsi_j\|_{L^3_x}^r+\|\partial_t\nabla\bfPhi_j\circ\bfPsi_j\|_{L^3_x}^r+\|\nabla^2\bfPhi_j\circ\bfPsi_j\partial_t\bfPsi_j\|^r_{L^3_x}\big)\|\bfU_j\|_{W^{1,p}_x}^r\dt\\
&+\sum_{j=1}^\ell \int_{\mathcal I}
\big(\|\partial_t\eta\|^r_{L^\infty_y}+\|\partial_t\nabla\eta\|^r_{L^1_y}\big)\|\bfU_j\|_{W^{1,p}_x}^r\dt\\
&\lesssim\sum_{j=1}^\ell \int_{\mathcal I}
\big(\|\partial_t\bfPhi_j\|_{W^{1,3}_x}^r+\|\nabla^2\bfPhi_j\|_{L^3}^r\|\partial_t\bfPsi_j\|^r_{L^\infty_x}+\|\partial_t\eta\|^r_{W^{1,q_0}_y}\big)\|\bfU_j\|_{W^{1,p}_x}^r\dt\\&\lesssim\sum_{j=1}^\ell \int_{\mathcal I}
\big(\|\partial_t\phi_j\|_{W^{2/3,3}_y}^r+\|\bfPhi_j\|_{W^{5/2,2}_x}^r\|\partial_t\phi_j\|^r_{L^{\infty}_y}+\|\partial_t\eta\|^r_{W^{1,q_0}_y}\big)\|\bfU_j\|_{W^{1,p}_x}^r\dt\\
&\lesssim\sum_{j=1}^\ell \int_{\mathcal I}
\big(\|\partial_t\phi_j\|_{W^{2/3,3}_y}^r+\|\phi_j\|_{W^{2,2}_x}^r\|\partial_t\eta\|^r_{W^{1,q_0}_y}+\|\partial_t\eta\|^r_{W^{1,q_0}_y}\big)\|\bfU_j\|_{W^{1,p}_x}^r\dt\\
&\lesssim\sum_{j=1}^\ell \int_{\mathcal I}
\big(1+\|\eta\|_{W^{2,2}_y}^r\big)\big(1+\|\partial_t\eta\|^r_{W^{1,q_0}_y}\big)\|\bfU_j\|_{W^{1,p}_x}^r\dt
\end{align*}
using also \eqref{est:ext}, the Sobolev embeddings $W^{1,2}_y\hookrightarrow W^{2/3,3}_y$ and $W^{5/2,2}_x\hookrightarrow W^{2,3}_x$ as well as \eqref{lem:9.4.1}. Since $\eta\in L^ \infty_tW^{2,2}_y$ by assumption
we can argue as in \eqref{eq:0301}
to obtain
\begin{align*}
\int_{\mathcal I}(R)_4^r\dt\lesssim \delta(T)\int_{\mathcal I}\|\bff\|_{L^p_x}^r\dt.
\end{align*}
If $p<\frac{3}{2}$ an analogous chain leads to
\begin{align*}
\int_{\mathcal I}(R)_4^r\dt&\lesssim\sum_{j=1}^\ell \int_{\mathcal I}
\|\partial_t\bfB_j^\top:\nabla\bfV_j\|_{L^{1}_x}^r\dt\lesssim\sum_{j=1}^\ell \int_{\mathcal I}
\|\partial_t\bfB_j\|_{L^3_x}^r\|\nabla\bfV_j\|_{L^{3/2}_x}^r\dt\\
&\lesssim\sum_{j=1}^\ell \int_{\mathcal I}
\big(1+\|\eta\|_{W^{2,2}_y}^r\big)\big(1+\|\partial_t\eta\|^r_{W^{1,q_0}_y}\big)\|\bfU_j\|_{W^{1,3/2}_x}^r\dt\\
&\lesssim\sum_{j=1}^\ell \int_{\mathcal I}
\big(1+\|\partial_t\eta\|^r_{W^{1,q_0}_y}\big)\|\bfU_j\|_{W^{1,3/2}_x}^r\dt.
\end{align*}
Now we use the embedding
\begin{align*}
W^{1,r}(\mathcal I,L^p(\mathbb H))\cap L^r(\mathcal I,W^{2,p}(\mathbb H))\hookrightarrow W^{\frac{3}{2p'},r}(\mathcal I,W^{\frac{3-p}{p},p}(\mathbb H))\hookrightarrow L^{\overline r}(\mathcal I,W^{1,3/2}(\mathbb H)),
\end{align*}
where $\overline r=\infty$ if $r>\frac{2p'}{3}$, $\overline r<\infty$ arbitrary if $r=\frac{2p'}{3}$ and $\overline r=\frac{2rp'}{2p'-3r}$ if $r<\frac{2p'}{3}$. This yields by \eqref{est:stokeswhole}, the assumptions on $\eta$ and $r_0>\max\{r,2p'/3\}$
\begin{align*}
\int_{\mathcal I}(R)_4^r\dt
&\lesssim \Big(\delta(T)+\|\eta\|^r_{L^{2p'/3}(\mathcal I,W^{1,q_0}(\omega))}\Big)\sum_{j=1}^\ell\big(\|\partial_t\bfU_j\|_{L^{r}(\mathcal I,L^p(\mathbb H))}^r+\|\bfU_j\|^r_{L^{r}(\mathcal I,W^{2,p}(\mathbb H))}\big)\\
&\lesssim \delta(T)\int_{\mathcal I}\|\bff\|_{L^{p}_x}^r\dt
\end{align*}
and thus the same estimate also in the case $p<\frac{3}{2}$.

In conclusion, we have shown
\begin{align}\label{eq:july28}
\int_{\mathcal I}\|\partial_t\mathscr R_1\bff\|_{L^{p}_x}^r\dt&\leq\,\tfrac{1}{4}\int_{\mathcal I}\|\bff\|_{L^{p}_x}^r\dt,
\end{align}
for $T$ and $\mathrm{Lip}(\partial\Omega_\eta)$ sufficiently small. As far 
as $\Delta\mathscr R_1\bff$ is concerned, we can use \eqref{eq:Br0} and argue as in \cite[proof of Theorem 3.1]{Br2} obtaining
\begin{align}\label{eq:july28B}
\int_{\mathcal I}\|\Delta\mathscr R_1\bff\|_{L^{p}_x}^r\dt&\leq\,\tfrac{1}{4}\int_{\mathcal I}\|\bff\|_{L^{p}_x}^r\dt
\end{align}
choosing $T$ and $\mathrm{Lip}(\partial\Omega_\eta)$  small enough.
Combining \eqref{eq:july24}, \eqref{eq:july28} and \eqref{eq:july28B} implies $\|\mathscr T\|\leq\tfrac{3}{4}$. Recalling \eqref{eq:IS}  we have shown the claim for $T$ sufficiently small, say $T=T_0\ll1$. It is easy to extend it to the whole interval.
Let $(\bfu,\pi)$ be the solution in $[0,T]$. In order to obtain a solution on the whole interval we consider a partition of unity
$(\psi_k)_{k=1}^K$ on $[0,T]$ such that $\mathrm{spt}(\psi_k)\subset(\alpha_k,\alpha_k+T_0]$ for some $(\alpha_k)_{k=2}^K\subset[0,T]$ and $\alpha_1=0$. The functions $(\bfu_k,\pi_k)=(\psi_k\bfu,\psi_k\pi)$ are the unique solutions to
\begin{align*}
\partial_t\bfu_k=\Delta \bfu_k-\nabla\pi_k+\bff+\psi_k'\bfu,\quad\Div\bfu_k=0,\quad\bfu_k|_{(\alpha_k,\alpha_k+T_0)\times\partial{\Omega_\eta}}=0,\quad \bfu_k(\alpha_k,\cdot)=0.
\end{align*}
Applying the result proved for the interval $[0,T_0]$ and noticing that $\bfu=\sum_{k=1}^K\bfu_k$ and $\pi=\sum_{k=1}^K\pi_k$ proves the claim in the general case.
\end{proof}

\section{Partial regularity}
\label{sec:blowup}

\subsection{Boundary suitable weak solutions}\label{sec:mainresults}
In the following we give a definition of boundary suitable weak solutions adapting the notation from \cite{SeShSo} and \cite{Br2}. 
For that purpose we fix two numbers $r_\ast\in(1,2)$ and $s_\ast\in(1,\frac{3}{2})$ such that $\frac{1}{r_\ast}+\frac{3}{2s_\ast}\geq 2$. The choice comes from the fact that the convective term $(\bfu\cdot\nabla)\bfu$ of a weak solution to \eqref{1} belongs to $L^{r_\ast}_tL^{s_\ast}_x$. Later on we will choose $r_\ast=5/3$ and, accordingly, $s_\ast=15/14$.
\begin{definition}[Boundary suitable weak solution] \label{def:weakSolution}
Let $\eta\in C(\overline{\mathcal I};C^{1}(\omega))$ with $\|\eta\|_{L^\infty_{t,y}}<L$.
Let $(\bff, \bfu_0)$ be a dataset such that
\begin{equation}
\begin{aligned}
\label{dataset}
&\bff \in L^{r_\ast}\big(\mathcal I; L^{s_\ast}(\Omega_\eta)\big),\quad 
\bfu_0\in W^{2,s_\ast}\cap W^{1,2}_{0,\mathrm{\Div}}(\Omega_{\eta(0)}).
\end{aligned}
\end{equation} 
We call the tuple
$(\bfu,\pi)$
a boundary suitable weak solution to the Navier--Stokes system \eqref{1} with data $(\bff, \bfu_0)$ provided that the following holds:
\begin{itemize}
\item[(a)] The velocity field $\bfu$ satisfies
\begin{align*}
 \bfu \in L^\infty \big(\mathcal I; L^2(\Omega_\eta) \big)\cap  L^2 \big(\mathcal I; W^{1,2}_{0,\Div}(\Omega_\eta) \big) \cap L^{r_\ast}(\mathcal I;W^{2,s_\ast}(\Omega_\eta))\cap W^{1,r_\ast}(\mathcal I;L^{s_\ast}(\Omega_\eta)).
\end{align*}
\item[(b)] The pressure $\pi$ satisfies
$$\pi\in L^{r_\ast}(\mathcal I;W^{1,s_\ast}_\perp(\Omega_\eta)).$$
\item[(c)] We have
\begin{align*}
\partial_t \bfu+(\bfu\cdot\nabla)\bfu&=\Delta\bfu-\nabla \pi+\bff,\quad \Div\bfu=0,\quad\bfu|_{\partial\Omega_\eta}=0,\quad \bfu(0,\cdot)=\bfu_0,
\end{align*}
a.a. in $\mathcal I\times\Omega_\eta$.
\item[(d)] for any $\zeta\in C_c^\infty(\mathcal I\times\R^3)$ with $\zeta\geq0$
 and $\mathrm{spt}(\zeta) \cap (\mathcal I\times \partial\Omega_\eta)\neq\emptyset$ 
the local energy inequality 
\begin{equation}\label{energylocal0}
\begin{split}
\int_ {\mt_\eta}\frac{1}{2}&\zeta\big| \bfu(t)\big|^2\dx+\int_0^t\int_ {\mt_\eta}\zeta|\nabla \bfu |^2\dx\ds\\& \leq\int_0^t\int_{\mt_\eta}\frac{1}{2}\Big(| \bfu|^2(\partial_t\zeta+\Delta\zeta)+\big(|\bfu|^2+2\pi\big)\bfu\cdot\nabla\zeta\Big)\dxs+\int_0^t\int_{\mt_\eta}\zeta\bff\cdot\bfu\dxs.
\end{split}
\end{equation}
holds.
\end{itemize}
\end{definition}
The next theorem shows that, under suitable assumptions on $\eta$, there a boundary suitable weak solution to \eqref{1}.
\begin{theorem}\label{thm:existence}
 Suppose that $\eta\in C(\overline{\mathcal I};B^{\theta}_{\varrho,s_*}\cap W^{2,2}\cap C^{1}(\omega))$ for some $\varrho\geq 2s_*/(s_*-1)$ and $\theta>2-1/s_*$, that $\sup_{\mathcal I} \mathrm{Lip}(\partial\Omega_{\eta(t)})$ is sufficiently small and that \eqref{eq:L0} holds. Suppose further that $\partial_t\eta\in L^{r_0}(\mathcal I; W^{1,q_0}(\omega))$ for some $r_0>\max\{2,r_*,2s_\ast'/3\}$ and $q_0>2$. Then there is a boundary suitable weak solution $(\bfu,\pi)$ to the Navier--Stokes equations \eqref{1} 
in the sense of Definition \ref{def:weakSolution}.
\end{theorem}
\begin{proof}
We consider an approximate system, where the convective term and the boundary function are regularised. We construct a solution to the system on $\Omega_{\mathfrak{r}^\xi\xi^\xi}$ with material derivation  $\partial_t +\mathcal{R}^\xi\bfu^\xi\cdot\nabla$ where $\xi>0$ is a fixed regularization kernel. Here $\mathcal R^\xi$ extends functions to $\mathcal I\times\R^3$ and mollifiers in space-time.
We denote by $\mathfrak r^\xi$
a regularisation operator acting on the periodic functions defined on $\omega$ composed with a temporal regularisation on $\mathcal I$. For $0<\xi\ll1$ we are looking for a function
$\bfu_\xi\in L^ \infty(I,L^2(\Omega_{\mathfrak r^\xi\eta}))\cap  L^2(I,W^{1,2}_{0,\Div}(\Omega_{\mathfrak r^\xi \eta}))$ satisfying
\begin{align}\begin{aligned}
\int_{\mathcal I}  \frac{\mathrm{d}}{\dt}\int_{\Omega_{\mathfrak{r}^\xi\eta(t)}}\bfu_\xi  \cdot {\bfvarphi}\dx
&\dt 
=\int_{\mathcal I}  \int_{\Omega_{\mathfrak{r}^\xi\eta(t)}}\big(  \bfu_\xi\cdot \partial_t  {\bfvarphi} -\tfrac{1}{2}(\mathcal{R}^\xi\bfu_\xi\cdot\nabla)\bfu_\xi \cdot  {\bfvarphi}  +\tfrac{1}{2}(\mathcal{R}^\xi\bfu_\xi\cdot\nabla){\bfvarphi} \cdot\bfu_\xi  \big) \dx\dt
\\&
-  \int_{\mathcal I}  \int_{\Omega_{\mathfrak{r}^\xi\eta(t)}}
\big(\mu \nabla \bfu_\xi :\nabla {\bfvarphi}- \mathcal{R}^\xi\bff\cdot{\bfvarphi} \big) \dx\dt
 \end{aligned}\label{distrho}
\end{align}
for all  $\bfvarphi \in  C^\infty(\overline{I}\times \R^3; \R^3)$ with ${\bfvarphi}(T,\cdot)=0$, $\Div{\bfvarphi}=0$ and $\bfvarphi|_{\partial\Omega_{\mathfrak{r}^\xi\eta}}= 0$ as well as $\bfu_ \xi(0)=\mathcal{R}^\xi\bfu_0$.
The existence of $\bfu_\xi$ can be shown
as in \cite[Prop. 3.4]{LeRu} (our situation is much easier since the boundary is given and the boundary conditions are trivial). It also satisfies a global energy inequality which implies 
\begin{align}\label{eq:2911}
\sup_{\mathcal I}\int_{\Omega_{\mathfrak r^\xi \eta}}|\bfu_\xi|^2 \dx+\int_{\mathcal I}\int_{\Omega_{\mathfrak r^\xi \eta}}|\nabla\bfu_\xi|^2 \dxt\leq\,c
\end{align}
uniformly in $\xi$.
One can understand 
\eqref{distrho} as Sokes system with right-hand side
\begin{align*}
\bfg_\xi:=-\frac{1}{2}(\mathcal{R}^\xi\bfu_\xi\cdot\nabla)\bfu_\xi-\frac{1}{2}(\bfu_\xi\cdot\nabla)\mathcal{R}^\xi\bfu_\xi+\mathcal{R}^\xi\bff
\end{align*}
in the moving domain $\Omega_{\mathfrak r^\xi \eta}$. Since $ \bfg_\xi$ belongs to $L^2(\mathcal I,L^2(\Omega_{\mathfrak r^\xi \eta}))$ and the initial datum
$\mathcal{R}^\xi\bfu_0$ belongs to $L^2(\Omega_{\mathfrak r^\xi \eta(0)})$ (with norms depending on $\xi$) we can apply Theorem \ref{thm:stokesunsteadymoving} with $p=r=2$ obtaining
\begin{align*}
\bfu_\xi\in L^2(\mathcal I,W^{2,2}(\Omega_{\mathfrak r^\xi \eta})),\quad \partial_t \bfu_\xi \in L^2(\mathcal I,L^2(\Omega_{\mathfrak r^\xi \eta})),\quad \pi_\xi\in L^2(\mathcal I,W^{1,2}(\Omega_{\mathfrak r^\xi \eta})),
\end{align*}
where $\pi_\xi$ is the associated pressure function.
This means we have a strong solution to \eqref{distrho}. In particular, we can test with 
$\zeta\bfu_\xi$, where $\zeta$ is a localisation function as required in \eqref{energylocal0}. This yields a version of \eqref{energylocal0} for $\bfu_ \xi$ in the domain  $\Omega_{\mathfrak r^\xi \eta}$.

The aim is now to pass to the limit in $\xi$.
The operator $\mathfrak r^\xi$ converges on all Besov spaces with $p<\infty$. 
Furthermore, it does not expand the $W^{1,\infty}$-norm, which is sufficient to apply 
 Theorem \ref{thm:stokesunsteadymoving} obtaining
 \begin{align}\label{eq:2505}
 \begin{aligned}
\|\partial_t\bfu_\xi\|_{L^{r_\ast}(\mathcal I;L^{s_\ast}(\Omega_{\mathfrak r^\xi \eta}))}&+\|\bfu_\xi\|_{L^{r_\ast}(\mathcal I;W^{2,s_\ast}(\Omega_{\mathfrak r^\xi \eta}))}
+\|\pi_\xi\|_{L^{r_\ast}(\mathcal I;W^{1,s_\ast}(\Omega_{\mathfrak r^\xi \eta}))}\\&\lesssim \|\mathcal{R}^\xi\bff\|_{L^{r_\ast}(\mathcal I;L^{s_\ast}(\Omega_{\mathfrak r^\xi \eta})}+\|(\mathcal{R}^\xi\bfu_\xi\cdot\nabla)\bfu_\xi\|_{L^{r_\ast}(\mathcal I;L^{s_\ast}(\Omega_{\mathfrak r^\xi \eta}))}\\&+\|(\bfu_ \xi\cdot\nabla)\mathcal{R}^\xi\bfu_\xi\|_{L^{r_\ast}(\mathcal I;L^{s_\ast}(\Omega_{\mathfrak r^\xi \eta}))}+\|\mathcal{R}^\xi\bfu_0\|_{W^{2,s_\ast}(\Omega_{\mathfrak r^\xi \eta(0)})}
\end{aligned}
\end{align}
uniformly in $\xi$. By \eqref{eq:2911}, our assumptions on the data as well as the properties of $\mathcal R^\xi$ the right-hand side is uniformly bounded.
With \eqref{eq:2505} at hand, we obtain (after passing to a subsequence) limit objects with the claimed regularity as well as compactness of $\bfu_\xi$ and can pass to the limit in the momentum equation and local energy inequality.
Note that this requires also the uniform convergence of $\mathfrak r^\xi\eta$ to $\eta$ which follows from the properties of $\mathfrak r^\xi$.
\end{proof}
The following theorem is the main result of this paper.
\begin{theorem}\label{thm:main}
 Suppose that $\eta\in C(\overline{\mathcal I};B^{\theta}_{p,p}\cap W^{2,2}\cap C^{1}(\omega))$ for some $p>\frac{15}{4}$ and some $\theta>2-1/p$, that $\sup_{\mathcal I} \mathrm{Lip}(\partial\Omega_{\eta(t)})$ is sufficiently small and that \eqref{eq:L0} holds. Suppose further that $\partial_t\eta\in L^{3}(\mathcal I; W^{1,q_0}(\omega))$ for some $q_0>2$. 
There is a number $\varepsilon_{0}>0$ such that the following holds. Let $(\bfu,\pi)$ be a boundary suitable weak solution to the Navier--Stokes system \eqref{1} in the sense of Definition \ref{def:weakSolution}.
 Let $(t_0,x_0)\in \mathcal I\times\partial\Omega_\eta$ be such that 
\begin{align}
r^{-2}\int_{t_0-r^2}^{t_0+r^2}\int_{\Omega_\eta\cap \mathcal B_r(x_0)}|\bfu|^3\dxt+\bigg(r^{-5/3}\int_{t_0-r^2}^{t_0+r^2}\int_{ \Omega_\eta\cap \mathcal B_r(x_0)}|\pi|^{5/3}\dx\dt\bigg)^{\frac{9}{5}}<\varepsilon_0
\end{align}
for some $r\ll1$.
Then we have $\bfu\in C^{0,\alpha}(\overline{\mathcal U}(t_0,x_0)\cap\overline{\mathcal I}\times \overline{\Omega}_\eta)$ for some $\alpha>0$ and a neighbourhood $\mathcal U(t_0,x_0)$ of $(t_0,x_0)$.
\end{theorem}
%
%
The proof of Theorem \ref{thm:main} will be given in the next subsection.
Denoting by $\mathcal H^s_{\mathrm{para}}$ the $s$-dimensional parabolic Hausdorff measure we obtain the following corollary from Theorem \ref{thm:main} concerning the size of the singular set.
\begin{corollary}\label{thm:main'}
Suppose the assumptions from  Theorem \ref{thm:main} hold.
Then there is a solution $(\bfu,\pi)$ to the Navier--Stokes equations \eqref{1} and a closed set $\Sigma\subset \mathcal I\times \partial\Omega_\eta$ with $\mathcal H^{5/3}_{\mathrm{para}}(\Sigma)=0$ such that for any $(t_0,x_0)\in \mathcal I\times \partial\Omega_\eta \setminus\Sigma$ we have $\bfu\in C^{0,\alpha}(\overline{\mathcal U}(t_0,x_0)\cap \overline{\mathcal{I}}\times\overline{\Omega}_\eta)$ for some $\alpha>0$ and a neighbourhood $\mathcal U(t_0,x_0)$ of $(t_0,x_0)$.
\end{corollary}
\begin{remark}
Corollary \ref{thm:main'} is a counterpart of \cite{Br2}, where fixed domains are considered. As there the Hausdorff-estimate for the singular set is worse than for smooth boundaries from \cite{SeShSo}. For sufficiently smooth functions $\eta$ we expect that the result from \cite{SeShSo} can be reproduced.
\end{remark}
\begin{remark}\label{rem:assdteta}
Taking the existence of a boundary suitable weak solution for granted, one can weaken the assumption
$\partial_t\eta\in L^{3}(\mathcal I; W^{1,q_0}(\omega))$ in Theorem \ref{thm:main} to $\partial_t\eta\in L^{3}(\mathcal I; L^{3}\cap W^{1,1}(\omega))$. This becomes evident from the estimates in the proof of Lemma \ref{Lemma}, in particular the proof of  \eqref{eq:regvm}. However, this is not sufficient to apply Theorem \ref{thm:stokesunsteadymoving} (which is crucial in the proof of Theorem \ref{thm:existence})
and hence the existence of such a solution is unclear.
\end{remark}

\subsection{The perturbed system}\label{sec:pert}

Let $\eta\in C(\overline{\mathcal I};W^{2,2}\cap C^{1}(\omega))$ be the function describing the boundary of $\Omega_\eta$ in accordance with Section \ref{ssec:geom} and suppose that \eqref{eq:L0} holds. 
 Let $(\bfu,\pi)$ be a boundary suitable weak solution to \eqref{1} in the sense of \ref{def:weakSolution}. For $(t_0,x_0)\in \mathcal I\times\partial\Omega_\eta$ we define functions $\phi$, $\bfPhi$, $\bfPsi$ and $\mathscr V$ (with matrix $\mathcal Q$) as in Section \ref{sec:repara}.
 We set $\overline \pi=\pi\circ\mathscr V\circ\bfPhi$, $\overline{\bfu}=\bfu\circ\mathscr V\circ\bfPhi$ and  $\overline{\bff}=\bff\circ\mathscr V\circ\bfPhi$. We also introduce
\begin{align}\label{eq:AB}
\bfA_\phi&=J_\phi\big(\nabla \bfPsi\circ\bfPhi\big)^{\top}\nabla \bfPsi\circ\bfPhi,\quad\bfB_\phi=J_\phi\nabla \bfPsi\circ\bfPhi\mathcal Q^\top,
\end{align}
where $J_\phi=\mathrm{det}(\nabla\bfPhi)$.
We see that $(\overline\bfu,\overline\pi)$ is a solution to the system
\begin{align}\label{momref}
J_{\phi}\partial_t\overline\bfu+\Div\big(\bfB_{\phi}\overline\pi\big)-\Div\big(\bfA_{\phi}\nabla\overline\bfu\big)&=- J_\phi\nabla\overline\bfu\,\partial_t(\bfPsi\circ\mathscr V^{-1})\circ\bfPhi-\bfB_\phi\nabla\overline\bfu\,\overline\bfu+J_\phi\overline\bff,\\
\label{divref}\bfB_{\phi}^\top:\nabla\overline\bfu=0,\quad\overline\bfu|_{\partial \mathcal B_1^+\cap\partial\mathbb H}&=0,
\end{align}
a.a. in $Q_1^+$. 
Note that it may be necessary to translate and scale the coordinates in space-time to arrive at a system posed in $Q_1^+:=Q_1(0,0)$ (rather than in $ Q_r(t_0,x_0)$ for some $r>0$, $t_0\in \mathcal I$ and $x_0\in\partial\mathbb H$).
 Similarly, we can transform the local energy inequality leading to
\begin{equation}\label{energylocal}
\begin{split}
\int_ {\mt_\eta}\frac{1}{2}J_\varphi\zeta\big| \overline\bfu(t)\big|^2\dx&+\int_0^t\int_ {\mt_\eta}\zeta\bfA_\varphi\nabla \overline\bfu:\nabla \overline\bfu\dx\ds\\& \leq\int_0^t\int_{\mt_\eta}\frac{1}{2} |\overline\bfu|^2J_\varphi\nabla\zeta\partial_t(\bfPsi\circ\mathscr V^{-1})\circ\bfPhi\dx\ds+\int_0^t\int_{\mt_\eta}\frac{1}{2}J_\varphi |\overline\bfu|^2\partial_t\zeta\dx\ds\\&+\int_0^t\int_{\mt_\eta}\frac{1}{2}J_\varphi| \overline\bfu|^2\Delta\bfPsi\circ\bfPhi\cdot\nabla\zeta\dx\ds+\int_0^t\int_{\mt_\eta}\frac{1}{2}| \overline\bfu|^2\bfA_\varphi:\nabla^2\zeta\dx\ds\\&+\int_0^t\int_{\mt_\eta}\frac{1}{2}\big(|\overline\bfu|^2+2\overline\pi\big)\overline\bfu\cdot\bfB_\varphi\nabla\zeta\dx\ds+\int_0^t\int_{\mt_\eta}J_\varphi\zeta\overline\bff\cdot\overline\bfu\dxs
\end{split}
\end{equation}
for any $\zeta\in C^\infty_c(Q_1)$ with $\zeta\geq0$. Note that first term on the right-hand side appears because of the moving boundary and vanishes for cylindrical domains.

\begin{definition}[Boundary suitable weak solution perturbed system] \label{def:weakSolutionflat}
Let $(\overline\bff, \overline\bfu_0)$ be a dataset such that
\begin{equation}
\begin{aligned}
\label{dataset}
&\overline\bff \in L^{r_\ast}\big(\mathcal I_1; L^{s_\ast}(\mathcal B_1^+))\big),\quad 
\overline\bfu_0\in W^{2,s_\ast}\cap W^{1,2}_{0,\mathrm{\Div}}(\mathcal B_1^+).
\end{aligned}
\end{equation} 
We call the tuple
$(\overline\bfu,\overline\pi)$
a boundary suitable weak solution to the perturbed Navier--Stokes system \eqref{momref} with data $(\overline\bff, \overline\bfu_0)$ provided that the following holds:
\begin{itemize}
\item[(a)] The velocity field $\overline\bfu$ satisfies
\begin{align*}
 \overline\bfu \in L^\infty \big(\mathcal I_1; L^2(\mathcal B_1^+) \big)\cap  L^2 \big(\mathcal I_1; W^{1,2}_{\Div}(\mathcal B_1^+) \big) \cap L^{r_\ast}(\mathcal I_1;W^{2,s_\ast}(\mathcal B_1^+))\cap W^{1,r_\ast}(\mathcal I_1;L^{s_\ast}(\mathcal B_1^+)).
\end{align*}
\item[(b)] The pressure $\overline\pi$ satisfies
$$\overline\pi\in L^{r_\ast}(\mathcal I_1;W^{1,s_\ast}_\perp(\mathcal B_1^+)).$$
\item[(c)] The system \eqref{momref}--\eqref{divref}
holds a.a. in $Q^+_1$.
\item[(d)] The local energy inequality \eqref{energylocal} holds for any $\zeta\in C_c^\infty(Q_1)$ with $\zeta\geq0$.
\end{itemize}
\end{definition}
In the bulk of our partial regularity proof in Section \ref{sec:blowup} we will compare a boundary suitable weak solution of the perturbed Navier--Stokes system with a solution of the perturbed Stokes system. Hence we analyse the latter in the following.

In analogy with \eqref{momref}--\eqref{divref} the perturbed Stokes system in $Q_1^+$ is given by
\begin{align}\label{eq:pertstokes}
\begin{aligned}
J_{\phi}\partial_t\overline\bfu+\Div\big(\bfB_{\phi}\overline{\pi}\big)-\Div\big(\bfA_{\phi}\nabla\overline\bfu\big)&=\overline \bfg,\\
\bfB_{\phi}^\top:\nabla\overline{\bfu}=\overline h,\quad\overline\bfu|_{\mathcal B^+_{1}\cap\partial\mathbb H}&=0,\quad\overline{\bfu}(-1,\cdot)=0,
\end{aligned}
\end{align}
where $\bfA_\phi$ and $\bfB_\phi$ are given in accordance with \eqref{eq:AB} for a given functions $\phi:\mathcal I_1\times \R^2\rightarrow\R$, an orthogonal matrix $\mathcal Q\in\R^{3\times 3}$ (depending continuously on time) and $\overline \bfg$ and $\overline h$ are given data.

We will also need interior estimates for points close to the boundary which is why we consider 
in analogy to \eqref{eq:pertstokes} the system
 \begin{align}\label{eq:pertstokesint}
\begin{aligned}
J_{\phi}\partial_t\overline\bfu+\Div\big(\bfB_{\phi}\overline{\pi}\big)-\Div\big(\bfA_{\phi}\nabla\overline\bfu\big)&=\overline{\bfg},\\
\bfB_{\phi}^\top:\nabla\overline{\bfu}=\overline h,\quad\overline\bfu|_{\mathcal B_{1}}&=0,\quad\overline{\bfu}(-1,\cdot)=0,
\end{aligned}
\end{align}
in $Q_1$. 

The following results for the systems \eqref{eq:pertstokes} and \eqref{eq:pertstokesint} follow now along the lines
of \cite[Lemma 4.2-4.4]{Br2} provided we have uniformly in time
\begin{align}\label{eq:Br}
\phi\in\mathcal M^{2-1/p,p}(\R^{2})(\delta),\quad \|\phi\|_{W^{1,\infty}_y}\leq \delta,
\end{align}
for some sufficiently small $\delta$. This is a direct consequence of \eqref{eq:Br0} (extending the function $\phi$ to $\R^2$ by means of a standard extension operator) provided we have $\eta\in C(\overline{\mathcal I};B^{\theta}_{\varrho,p}\cap W^{1,\infty}(\omega))$ for some $\theta>2-1/p$ and $\sup_{\mathcal I} \mathrm{Lip}(\partial\Omega_{\eta(t)})$ is sufficiently small. Note that this also implies that the matrix $\mathcal Q$ is bounded in time, cf. \eqref{eq:V}. 
\begin{lemma}\label{lem:31}
Let $p,r\in(1,\infty)$ be given. 
Suppose that $\phi:\mathcal I_1\times\R^2\rightarrow\R$ satisfies \eqref{eq:Br} and $\mathcal Q$ belongs to the class $L^\infty(\overline{\mathcal I}_1)$ with values in the set of orthogonal $3 \times 3$-matrices.
 Assume further that $\overline \bfg\in L^r(\mathcal I_1;L^{p}(\mathcal B_1^+))$ and $\overline h\in L^r(\mathcal I_1;W^{1,p}\cap L^p_\perp(\mathcal B_1^+))$ with $\partial_t \overline h\in L^r(\mathcal I_1;W^{-1,p}(\mathcal B_1^+))$ and $\overline h(0,\cdot)=0$. Then there is a unique solution $(\overline\bfu,\overline\pi)$ to \eqref{eq:pertstokes} which satisfies
\begin{align}\label{eq:mainpara}
\begin{aligned}
\|\partial_t\overline\bfu\|_{L^r(\mathcal I_1;L^p(\mathcal B^+_{1}))}&+\|\overline\bfu\|_{L^r(\mathcal I_1;W^{2,p}(\mathcal B^+_{1}))}
+\|\overline\pi\|_{L^r(\mathcal I_1;W^{1,p}(\mathcal B^+_{1}))}\\&\lesssim \|\overline \bfg\|_{L^r(\mathcal I_1;L^p(\mathcal B^+_{1}))}+\|\nabla \overline h\|_{L^r(\mathcal I_1;L^{p}(\mathcal B^+_{1}))}+\|\partial_t \overline h\|_{L^r(\mathcal I_1;W^{-1,p}(\mathcal B^+_{1}))},
\end{aligned}
\end{align}
where the hidden constant only depends on $p,r$ and $\delta$.
\end{lemma}
\begin{lemma}\label{lem:32}
Let $p,q,r\in(1,\infty)$ with $q\geq p$ be given. 
Suppose that $\phi:\mathcal I_1\times\R^2\rightarrow\R$ satisfies \eqref{eq:Br} and $\mathcal Q$ belongs to the class $L^\infty(\overline{\mathcal I}_1)$ with values in the set of orthogonal $3 \times 3$-matrices.
Assume further that $\overline\bfg\in L^r(\mathcal I_1;L^{q}(\mathcal B_1^+))$ and that $\overline h=0$.
The solution $(\overline\bfu,\overline\pi)$ to \eqref{eq:pertstokes} satisfies
\begin{align}\label{eq:mainpara'}
\begin{aligned}
\|\partial_t\overline\bfu\|_{L^r(\mathcal I_{1/2};L^q(\mathcal B^+_{1/2}))}&+\|\overline\bfu\|_{L^r(\mathcal I_{1/2};W^{2,q}(\mathcal B^+_{1/2}))}
+\|\overline\pi\|_{L^r(\mathcal I_{1/2};W^{1,q}(\mathcal B^+_{1/2}))}\\&\lesssim \|\overline\bfg\|_{L^r(\mathcal I_1;L^q(\mathcal B^+_{1}))}+\|\nabla \overline\bfu\|_{L^r(\mathcal I_1;L^{p}(\mathcal B^+_{1}))}+\|\overline\pi-(\overline\pi)_{\mathcal B_1^+}\|_{L^r(\mathcal I_1;L^{p}(\mathcal B^+_{1}))},
\end{aligned}
\end{align}
where the hidden constant only depends on $p,q,r$ and $\delta$.
\end{lemma}
\begin{lemma}\label{lem:31int}
Let $p,r\in(1,\infty)$ be given. 
Suppose that $\phi:\mathcal I_1\times\R^2\rightarrow\R$ satisfies \eqref{eq:Br} and $\mathcal Q$ belongs to the class $L^\infty(\overline{\mathcal I}_1)$ with values in the set of orthogonal $3 \times 3$-matrices.
Assume further that $\overline \bfg\in L^r(\mathcal I_1;L^{p}(\mathcal B_1))$ and $\overline h\in L^r(\mathcal I_1;W^{1,p}(\mathcal B_1))$ with $\partial_t \overline h\in L^r(\mathcal I_1;W^{-1,p}(\mathcal B_1))$ and $\overline h(0,\cdot)=0$. Then there is a unique solution $(\overline\bfu,\overline\pi)$ to \eqref{eq:pertstokesint} which satisfies
\begin{align}\label{eq:mainparaint}
\begin{aligned}
\|\partial_t\overline\bfu\|_{L^r(\mathcal I_1;L^p(\mathcal B_{1}))}&+\|\overline\bfu\|_{L^r(\mathcal I_1;W^{2,p}(\mathcal B_{1}))}
+\|\overline\pi\|_{L^r(\mathcal I_1;W^{1,p}(\mathcal B_{1}))}\\&\lesssim \|\overline\bfg\|_{L^r(\mathcal I_1;L^p(\mathcal B_{1}))}+\|\nabla \overline h\|_{L^r(\mathcal I_1;L^{p}(\mathcal B_{1}))}+\|\partial_t \overline h\|_{L^r(\mathcal I_1;W^{-1,p}(\mathcal B_{1}))},
\end{aligned}
\end{align}
where the hidden constant only depends on $p,r$ and $\delta$.
\end{lemma}
\begin{lemma}\label{lem:32int}
Let $p,q,r\in(1,\infty)$ with $q\geq p$ be given. 
Suppose that $\phi:\mathcal I_1\times\R^2\rightarrow\R$ satisfies \eqref{eq:Br} and $\mathcal Q$ belongs to the class $L^\infty(\overline{\mathcal I}_1)$ with values in the set of orthogonal $3 \times 3$-matrices.
Assume further that $\overline\bfg\in L^r(\mathcal I_1;L^{q}(\mathcal B_1))$ and that $\overline h=0$.
The solution $(\overline\bfu,\overline\pi)$ to \eqref{eq:pertstokesint} satisfies
\begin{align}\label{eq:mainpara'int}
\begin{aligned}
\|\partial_t\overline\bfu\|_{L^r(\mathcal I_{1/2};L^q(\mathcal B_{1/2}))}&+\|\overline\bfu\|_{L^r(\mathcal I_{1/2};W^{2,q}(\mathcal B_{1/2}))}
+\|\overline\pi\|_{L^r(\mathcal I_{1/2};W^{1,q}(\mathcal B_{1/2}))}\\&\lesssim \|\overline \bfg\|_{L^r(\mathcal I_1;L^q(\mathcal B_{1}))}+\|\nabla \overline\bfu\|_{L^r(\mathcal I_1;L^{p}(\mathcal B_{1}))}+\|\overline\pi-(\overline\pi)_{\mathcal B_1}\|_{L^r(\mathcal I_1;L^{p}(\mathcal B_{1}))},
\end{aligned}
\end{align}
where the hidden constant only depends on $p,q,r$ and $\delta$.
\end{lemma}

\subsection{The blow-up lemma}
We consider a boundary suitable weak solution $(\overline\bfu,\overline\pi)$ to the perturbed Navier--Stokes system \eqref{momref}--\eqref{divref} with forcing $\overline\bff$ as defined in Definition
\ref{def:weakSolutionflat}.
We define
 the excess-functional as
\begin{align*}
\mathscr E_r^{t_0,x_0} (\overline\bfu,\overline\pi)  &:=  \dashint_{Q^+_{r}(t_0,x_0)} |\overline\bfu|^3 \dy\ds+ r^3\bigg(\dashint_{Q^+_{r}(t_0,x_0)}  |\overline\pi - (\overline\pi)_{\mathcal B^+_r(x_0)}|^{5/3} \dy\ds\bigg)^{\frac{9}{5}}
\end{align*}
and use the short-hand notation $\mathscr E_r(\overline\bfu,\overline\pi)=:\mathscr E_r^{0,0} (\overline\bfu,\overline\pi)$.
The following lemma is the bulk of the partial regularity proof and the claim of Theorem \ref{thm:main} will follow in a standard manner.
\begin{lemma}
\label{Lemma}
 Suppose that $\eta\in C(\overline{\mathcal I}_1;B^{\theta}_{p,p}\cap W^{2,2}(\omega))$ for some $p>\frac{15}{4}$ and some $\theta>2-\frac{1}{p}$, that $\sup_I \mathrm{Lip}(\partial\Omega_{\eta(t)})$ is sufficiently small and that \eqref{eq:L0} holds.\footnote{Note that under this conditions $B^{\theta}_{p,p}(\omega)\hookrightarrow C^{1}(\omega)$.} Suppose further that $\partial_t\eta\in L^{3}(\mathcal I; W^{1,q_0}(\omega))$ for some $q_0>2$
and $\overline\bff\in L^{p}(\mathcal I_1;L^{p}(\mathcal B_1^+))$. 
For any $\tau\in(0,1/2)$ there exist constants
$\varepsilon>0$ (small) and $C_\ast>0$ (large) such the following implication is true for any
tuple $(\overline\bfu,\overline\pi)$ which is a boundary suitable weak solution to \eqref{momref}--\eqref{divref} in $\mathcal Q_1^+$ in the sense of Definition \ref{def:weakSolution}: Suppose that
\begin{equation}
\label{I1}
\mathscr E_1(\overline\bfu,\overline\pi) + 
\bigg(\dashint_{Q^+_1}  |\overline\bff|^{p} \dy\ds\bigg)^{\frac{3}{p}}
\leq\varepsilon,
\end{equation}
then we have
\begin{equation}
\label{I2}
\mathscr E_\tau (\overline\bfu,\overline\pi) \le C_\ast \tau^{2\alpha} \bigg(\mathscr E_1 (\overline\bfu,\overline\pi) +\bigg(\dashint_{Q^+_1}  |\overline\bff|^{p} \dy\ds\bigg)^{\frac{3}{p}}\bigg),
\end{equation}
where $\alpha=3\big(\frac{2}{5}-\frac{3}{2p}\big)$.
\end{lemma}
\begin{remark}
The proof of Lemma \ref{Lemma} is based on Lemmas \ref{lem:31}--\ref{lem:32int} which hold under relaxed assumption on the time-regularity of $\eta$ compared to Theorem \ref{thm:stokesunsteadymoving}. However, we still have to control the time derivative of local charts from Section \ref{sec:repara} and, in particular, \eqref{eq:dtphi} for a sequence of functions. Hence we require the same condition.
\end{remark}
\begin{proof}[Proof of Lemma \ref{Lemma}.] We argue by contradiction. Suppose there is $\tau \in (0,\tfrac{1}{2})$,  a sequence
$(\eta_m)$ such that
 uniformly in $m$ (for some sufficiently small $\delta>0$, $L_0\in(0,L)$ and $\kappa_0>0$ with $\bfvarphi_{\eta_m}$ given by \eqref{eq:bfvarphi})
\begin{align}\label{eq:etam}
\eta_m \in C(\overline{\mathcal I}_1;&B^{\theta}_{p,p}\cap W^{2,2}(\omega))\cap W^{1,3}(\mathcal I_1;W^{1,q_0}(\omega)),\quad \sup_{\mathcal I_1} \mathrm{Lip}(\partial\Omega_{\eta_m(t)})\leq\delta,\\
\label{eq:L0m}
&\|\eta_m\|_{L^\infty(\mathcal I_1\times \omega)}\leq L_0,\quad \inf_{\mathcal I_1\times\omega}\partial_1\bfvarphi_{\eta_m}\times\partial_2\bfvarphi_{\eta_m}\geq \kappa_0,
\end{align} 
and a sequence of boundary suitable weak solution $(\overline\bfu_m,\overline\pi_m)$ to the perturbed system in the sense of Definition \ref{def:weakSolutionflat} such that 
\begin{align}
\label{I3}
\lambda_m^3:=& \mathscr E_1(\overline\bfu_m,\overline\pi_m)+\|\overline\bff\|^3_{ L^{p}(\mathcal I_1;L^{p}(\mathcal B_1^+))} \to 0, \quad \ m \to \infty, \\
\label{I4}
&\mathscr E_\tau(\overline\bfu_m,\overline\pi_m) >  \tfrac{1}{2} \lambda^3_m .
\end{align}
Clearly, \eqref{eq:etam} implies
\begin{align}\label{eq:etam'}
\begin{aligned}
&\eta_m\rightharpoonup^\ast\eta\quad\text{in}\quad  L^\infty(\mathcal I_1;B^{\theta}_{p,p}\cap W^{2,2}(\omega))\quad \sup_{\mathcal I_1} \mathrm{Lip}(\partial\Omega_{\eta(t)})\leq\delta,\\
 &\qquad\qquad\partial_t\eta_m\rightharpoonup\partial_t\eta\quad\text{in}\quad L^{3}(\mathcal I_1;W^{1,q_0}(\omega)),
\end{aligned}
\end{align}for a limit function $\eta$, at least after taking a subsequence. By interpolation,
this yields
\begin{align}\label{eq:etam''}
\eta\in  C(\overline{\mathcal I_1};B^{\theta'}_{p,p}(\omega))
\end{align}
for any $\theta'<\theta.$ 
In accordance with \eqref{eq:phi} we can compute $\phi_m$ deducing that uniformly in time (decreasing the value of $\delta$ if necessary)
\begin{align}\label{eq:varphim}
\|\phi_m\|_{W^{1,\infty}_y}+ \|\phi_m\|_{\mathcal M^{2-1/p,p}(\R^{2})}\leq \delta.
\end{align}
We obtain from \eqref{eq:varphim}
\begin{align}\label{eq:varphim2}
\nabla\phi_m\rightharpoonup^\ast\nabla\phi\quad \text{in}\quad L^{\infty}(\mathcal I_1\times\R^{2}),\quad \phi\in L^\infty(\mathcal I_1,\mathcal M^{2-1/p,p}(\R^{2})(\delta)),\quad \|\phi\|_{W^{1,\infty}_y}\leq \delta.
\end{align}
Hence \eqref{est:ext} and \eqref{eq:MS} yields
\begin{align}\label{eq:Phim2}
\nabla\bfPhi_m\rightharpoonup^\ast\nabla\bfPhi\quad \text{in}\quad L^{\infty}(\mathcal I_1\times\R^3),\quad \bfPhi\in L^\infty(\mathcal I_1,\mathcal M^{2,p}(\R^3)),
\end{align}
where $\bfPhi_m$ and $\bfPhi$ are defined in accordance with \eqref{eq:Phi}. Finally, \eqref{eq:detJ} and \eqref{eq:SMPhiPsi} imply for the inverse functions of $\bfPhi_m$ and $\bfPhi$
\begin{align}\label{eq:Psim2}
\nabla\bfPsi_m\rightharpoonup^\ast\nabla\bfPsi\quad \text{in}\quad L^{\infty}(\mathcal I_1\times\R^3),\quad \bfPsi\in L^\infty(\mathcal I_1,\mathcal M^{2,p}(\R^3)).
\end{align}
 Thanks to \eqref{eq:dtphi} and \eqref{eq:etam} we can control the time derivatives of $\phi_m$, $\bfPhi_m$ and $\bfPsi_m$. In particular, we have uniformly in $m$
\begin{align}\label{eq:dtphim}
\partial_t\phi_m\in L^3(\mathcal I_1;L^\infty(\mathcal B_1^+))
\end{align}
and thus
\begin{align}\label{eq:dtPsim}
\partial_t(\bfPsi_m\circ\mathscr V_m)\in L^3(\mathcal I_1;L^\infty(\mathcal B_1^+))
\end{align}
by \eqref{eq:dtQ}, \eqref{est:ext} and \eqref{eq:Psim2}.

We define in $\mathcal I_1\times\mathcal B^+_1$
\begin{align*}
\overline\bfv_m&:= \frac{1}{\lambda_m} \overline\bfu_m,\quad
\overline{\mathfrak q}_m:= \frac{1}{\lambda_m } \big[\overline\pi_m- (\overline\pi_m)_{\mathcal B_1^+}\big],\quad
\overline\bfg_m:=\frac{1}{\lambda_m}\overline\bff.
\end{align*}
We get from \eqref{I3} 
the relation
\begin{equation}
\label{I6} \dashint_{Q^+_1} |\overline\bfv_{m}|^3 \dz\ds+ \bigg(\dashint_{Q^+_1} |\overline{\mathfrak q}_{m}|^{5/3} \dz\ds\bigg)^{\frac{9}{5}}+ \bigg(\dashint_{Q^+_1} |\overline\bfg_{m}|^{p} \dz\ds\bigg)^{\frac{3}{p}} =  1,
\end{equation}
such that, after passing to a subsequence,
\begin{align}
\label{I8}
 \overline{\mathfrak{q}}_m &\rightharpoonup : \overline{\mathfrak q} \quad \text{in} \quad L^{5/3} (\mathcal I_1;L^{5/3}(\mathcal B^+_1)),\\
\overline\bfv_m &\rightharpoonup : \overline\bfv\quad\text{in}\quad  L^3(\mathcal I_1;L^3(\mathcal B^+_1))\label{I8'},\\
\overline\bfg_m &\rightharpoonup : \overline\bfg\quad\text{in}\quad  L^p(\mathcal I_1;L^p(\mathcal B^+_1))\label{I8''}.
\end{align}
Now, (\ref{I4}) reads after scaling
\begin{align}
\label{I7}
\dashint_{ Q_\tau^+} |\overline\bfv_{m} |^3 \dz\ds
 &+\tau^3\bigg(\dashint_{\mathcal I_1}\dashint_{\mathcal B_\tau^+}  |\overline{\mathfrak{q}}_m  - (\overline{\mathfrak q}_{m})_{\mathcal B_\tau^+} |^{5/3} \dz\ds\bigg)^{\frac{9}{5}}> C_{\ast} \tau^{2\alpha} ,
\end{align}
while the equation reads as
\begin{align}\label{eq:eq}
\begin{aligned}
J_{\phi_m}\partial_t\overline\bfv_m+\lambda_m\bfB_{\phi_m}\nabla\overline\bfv_m\,\overline\bfv_m&+\lambda_mJ_{\phi_m}\nabla\overline\bfv_m\partial_t(\bfPsi_m\circ\mathscr V_m^{-1})\circ\bfPhi_m\\&+\Div\big(\bfB_{\phi_m}\overline{\mathfrak q}_m\big)-\Div\big(\bfA_{\phi_m}\nabla\overline\bfv_m\big)=J_{\phi_m}\overline\bfg_m,\\
\bfB_{\phi_m}:\nabla{\overline\bfv}_m=0,\quad\overline\bfv_{m}|_{\mathcal B^+_{1}\cap\partial\mathbb H}&=0.
\end{aligned}
\end{align}
Note that the term $\lambda_mJ_{\phi_m}\nabla\overline\bfv_m\partial_t(\bfPsi_m\circ\mathscr V_m^{-1}))\circ\bfPhi_m$ on the left-hand side does not appear in \cite{Br2} but can be estimated accordingly thanks to \eqref{eq:dtPsim}.
By \eqref{energylocal} (replacing $\zeta$ by $\zeta^2$) we have
\begin{align*}
\int_ {\mathcal B^+_1}\frac{1}{2}J_{\phi_m}\zeta^2&\big| \overline\bfv_m(t)\big|^2\dx+\int_0^t\int_ {\mathcal B^+_1}\zeta^2\bfA_{\phi_m}\nabla \overline\bfv_m:\nabla \overline\bfv_m\dx\ds\\& \leq\int_0^t\int_{\mt}\frac{1}{2} |\overline\bfv_m|^2J_{\phi_m}\nabla\zeta^2\partial_t(\bfPsi_m\circ\mathscr V_m)\circ\bfPhi_m\dx\ds+\int_0^t\int_{\mathcal B^+_1}\frac{1}{2}J_{\phi_m} |\overline\bfv_m|^2\partial_t\zeta^2\dx\ds\\&+\int_0^t\int_{\mathcal B^+_1}\frac{1}{2}J_{\phi_m}| \overline\bfv_m|^2\Delta{\bfPsi_m}\circ\bfPhi_m\cdot\nabla\zeta^2\dx\ds+\int_0^t\int_{\mathcal B^+_1}\frac{1}{2}| \overline\bfv_m|^2\bfA_{\phi_m}:\nabla^2\zeta^2\dx\ds\\
&+\int_0^t\int_{\mathcal B^+_1}\frac{1}{2}\big(\lambda_m|\overline\bfv_m|^2+2\overline{\mathfrak q}_m\big)\overline\bfv_m\cdot\bfB_{\phi_m}\nabla\zeta^2\Big)\dx\ds+\int_0^t\int_{\mathcal B_1^+}J_{\phi_m}\zeta^2\overline\bfg_m\cdot\overline\bfv_m\dxs
\end{align*}
for all non-negative $\zeta\in C^\infty_c(\mathcal Q_1)$.  
All terms on the right-hand side can be estimate along the lines of \cite[Section 4]{Br2} but the first one.
Since $\overline{\bfv}_m\in L^3_{t,z}$ uniformly by \eqref{I7}
and $\partial_t(\bfPsi_m\circ\mathscr V_m)\in L^{3}_tL^{\infty}_z$ by \eqref{eq:dtPsim}
it is bounded as well.
We conclude
\begin{align}\label{eq:regvm}
\overline{\bfv}_m\in L^\infty\big(\mathcal I_{3/4};L^2\big(\mathcal B^+_{3/4}\big)\big)\cap L^2\big(\mathcal I_{3/4};W^{1,2}\big(\mathcal B_{3/4}^+\big)\big)
\end{align}
uniformly in $m$.
By Sobolev's inequality and \eqref{eq:eq} we have for $\bfphi\in C^\infty_c(\mathcal B_{3/4}^+)$
\begin{align*}
\int_{\mathcal B^+_1} \partial_t \overline{\bfv}_m\cdot \bfphi\dz&\leq \,c\,\Big(\|\overline{\bfv}_m\|_{L^2_z}\|\nabla\overline{\bfv}_m\|_{L^2_z}+\|\nabla\overline{\bfv}_m\|_{L^2_z}\|\partial_t(\bfPsi_m\circ\mathscr V_m^{-1})\circ\bfPhi_m\|_{L^2_z}\Big)\|\bfphi\|_{W^{2,2}_z}\\&+c\Big(\|\overline{\mathfrak{q}}_m\|_{L^{3/2}_z}+\|\nabla\overline{\bfv}_m\|_{L^2_z}+\|\overline\bfg_m\|_{L^2_z}\Big)\|\bfphi\|_{W^{2,2}_z}.
\end{align*}
This shows by \eqref{I4}, \eqref{eq:dtPsim} and \eqref{eq:regvm} the boundedness of 
\begin{align}\label{eq:dtvm}
\partial_t \overline{\bfv}_m\in L^{6/5}(\mathcal I_{3/4};W^{-2,2}(\mathcal B^+_{3/4})).
\end{align}
After passing to suitable subsequences we obtain
\begin{align}
\label{31}\partial_t \overline{\bfv}_m&\rightharpoonup \partial_t\overline{\bfv}\quad\text{in}\quad L^{6/5}\big(\mathcal I_{3/4};W^{-2,2} 
\big(\mathcal B^+_{3/4}\big)\big).
\end{align}
Now, \eqref{eq:regvm} and \eqref{eq:dtvm} imply
\begin{align}\label{eq:nablavmcomp}
\overline\bfv_m&\rightarrow \overline\bfv\quad\text{in}\quad L^3\big(\mathcal I_{3/4};L^{3}\big(\mathcal B_{3/4}^+\big)\big)
\end{align}
by the Aubin-Lions compactness theorem.
Recalling the definitions of $\bfA_{\phi_m}$ and $\bfB_{\phi_m}$ from \eqref{eq:AB} we are able to pass to the limit in \eqref{eq:eq} using the convergences \eqref{eq:Phim2}, \eqref{eq:Psim2}, \eqref{I8}, \eqref{I8'} \eqref{eq:regvm} and \eqref{31}. We obtain (in the sense of distributions on $Q^+_{3/4}$)
\begin{align}\label{eq:eqlimit}
\begin{aligned}
J_{\phi}\partial_t\overline\bfv+\Div\big(\bfB_{\phi}\overline{\mathfrak q}\big)-\Div\big(\bfA_{\phi}\nabla\overline\bfv\big)&=J_\phi\overline\bfg,\\
\bfB_{\phi}^\top:\nabla{\overline\bfv}=0,\quad\overline\bfv|_{\mathcal B^+_{3/4}\cap\partial\mathbb H}&=0.
\end{aligned}
\end{align}
Note that  $\lambda_mJ_{\phi_m}\nabla \overline\bfv_m\partial_t(\bfPsi_m\circ\mathscr V_m)\circ\bfPhi_m$ vanishes as $m\rightarrow\infty$ thanks to \eqref{eq:regvm} and \eqref{eq:dtPsim}.

Exactly as in \cite[Section 4]{Br2} (this requires the assumption $p>15/4$) we can derive a contradiction to
\eqref{I7} applying the regularity theory developed in Section \ref{sec:pert} to \eqref{eq:eq} and \eqref{eq:eqlimit}. This finishes the proof.
\end{proof}
\begin{proof}[Proof of Theorem \ref{thm:main}]
By assumption we have
\begin{align*}
r^{-2}\int_{t_0-r^2}^{t_0+r^2}\int_{\Omega_\eta\cap \mathcal B_r(x_0)}|\bfu|^3\dxt+\bigg(r^{-5/3}\int_{t_0-r^2}^{t_0+r^2}\int_{\Omega_\eta\cap \mathcal B_r(x_0)}|\pi|^{5/3}\dx\dt\bigg)^{\frac{9}{5}}<\varepsilon_0
\end{align*}
from which we deduce arguing as in \cite[Section 4]{Br2}
and applying a standard iteration procedure (see, e.g., \cite[Prop. 2.5]{EsSeSv})
\begin{align}
\widetilde{\mathscr E}_{\tau^k R}^{t_0,x_0}(\overline\bfu,\overline\pi)\lesssim \tau^{2\alpha k},\label{decay1}
\end{align}
where the access $\widetilde{\mathscr E}$ is given by
\begin{align*}
\widetilde{\mathscr E}_{r}^{t_0,x_0} (\overline\bfu,\overline\pi)  &:=  \dashint_{Q_{r}(t_0,x_0)\cap (I\times\Omega)} |\overline\bfu-(\overline\bfu)_{Q_{r}(t_0,x_0)\cap (I\times\Omega)}|^3 \dy\ds\\&+ r^3\bigg(\dashint_{\mathcal I_r(t_0)}\dashint_{\mathcal B^+_{r}(x_0)}  |\overline\pi - (\overline\pi)_{\mathcal B_r(x_0)\cap\Omega}|^{5/3} \dy\ds\bigg)^{\frac{9}{5}}
\end{align*}
and $\overline\bfu:=\bfu\circ \mathscr V\circ\bfPhi$ and $\overline\pi:=\pi\circ \mathscr V\circ\bfPhi$ solved the perturbed system, cf. Section. \ref{sec:pert}.
The decay estimate \eqref{decay1} holds in the centre point $(t_0,x_0)=(0,0)$ of $Q_1^+$ but also for points which are sufficiently close.
Furthermore, we can prove a version of Lemma \ref{Lemma} for interior points (one must replace in the proof Lemmas \ref{lem:31} and \ref{lem:32} by their interior counterparts Lemmas \ref{lem:31int} and \ref{lem:32int}). This finally shows
$\overline\bfu\in C^{0,\alpha}$ in a neighbourhood of $(0,0)$. Changing coordinates and using that $\bfu=\overline\bfu\circ\bfPsi\circ\mathscr V_m$, where $\bfPsi$ and $\mathscr V_m$ are Lipschitz uniformly in time, cf. \eqref{eq:detJ}) and \eqref{eq:V}, proves $\bfu\in C^{0,\alpha}(\overline{\mathcal U}(t_0,x_0)\cap\overline{\mathcal I}\overline{\Omega}_\eta)$ in a neighbourhood $\mathcal U(t_0,x_0)$ of $(t_0,x_0)$.
This finishes the proof of Theorem \ref{thm:main}.\end{proof}

%

\section*{Compliance with Ethical Standards}\label{conflicts}
\smallskip
\par\noindent
{\bf Conflict of Interest}. The author declares that he has no conflict of interest.

\smallskip
\par\noindent
{\bf Data Availability}. Data sharing is not applicable to this article as no datasets were generated or analysed
during the current study.

\end{document}